\documentclass[a4paper,11pt]{article}

\title{Absolutely indecomposable quasi-parabolic $G$-bundles and the multiplicity   of irreducible characters}

\author{  GyeongHyeon Nam \\ {\it
  Korea Institute for Advanced Study} \\{\tt  ghnam@kias.re.kr}}
 \date{}
 
\usepackage{amsmath, amsthm, amssymb, amsfonts, amsthm, mathrsfs}
\usepackage{textgreek, makecell, enumerate, bm, mathtools, thmtools, float, cite, wrapfig, multicol,relsize}
\usepackage[dvipsnames]{xcolor}
\usepackage{colonequals}
\usepackage{hyperref}
\usepackage{a4wide,xy}
\usepackage{tikz-cd}
  
\usepackage[nobysame, alphabetic]{amsrefs}

\let\oldmarginpar\marginpar
\renewcommand\marginpar[1]{\-\oldmarginpar[\raggedleft\footnotesize #1]
	{\raggedright\footnotesize #1}}

\makeindex

\usepackage{xparse}
\let\realItem\item 
\makeatletter
\NewDocumentCommand\myItem{ o }{%
   \IfNoValueTF{#1}%
      {\realItem}
      {\realItem[#1]\def\@currentlabel{#1}}
}
\makeatother

\usepackage{enumitem}    
\setlist[enumerate]{
    before=\let\item\myItem,       
    label=\textnormal{(\arabic*)}, 
    widest=(2')                    
}

\xyoption{all}
\input{xypic}

\numberwithin{equation}{section}

\makeatother
\theoremstyle{plain}
\newtheorem{thm}{Theorem}
\newtheorem{lem}[thm]{Lemma}

\newtheorem{conj}[thm]{Conjecture}

\newtheorem{cor}[thm]{Corollary}

\newtheorem*{assu}{Assumption}
\theoremstyle{definition}
\newtheorem{defe}[thm]{Definition}
 
\theoremstyle{remark}
\newtheorem{rem}[thm]{Remark}

\usepackage{enumitem}
\definecolor{red}{rgb}{1,0,0}
\definecolor{orange}{rgb}{1,0.5,0}
\definecolor{purple}{rgb}{.5,.2,.8}
\definecolor{blue}{rgb}{.2,.2,.8}
\definecolor{green}{rgb}{.4,.6,.4}

\newcommand{\nc}{\newcommand}
\newcommand{\rc}{\renewcommand}
\nc{\on}{\operatorname}
\nc{\Fq}{\mathbb{F}_q}
\nc{\fg}{\mathfrak{g}}
\nc{\ft}{\mathfrak{t}}
\nc{\fl}{\mathfrak{l}}
\nc{\fp}{\mathfrak{p}}
\nc{\fn}{\mathfrak{n}}
\nc{\Ad}{\mathrm{Ad}}
\nc{\fb}{\mathfrak{b}}

\rc{\AA}{\mathbb{A}}	
\nc{\BB}{\mathbb{B}}	
\nc{\CC}{\mathbb{C}}	
\nc{\DD}{\mathbb{D}}	
\nc{\EE}{\mathbb{E}}	
\nc{\FF}{\mathbb{F}}	
\nc{\GG}{\mathbb{G}}	
\nc{\HH}{\mathbb{H}}	
\nc{\II}{\mathbb{I}}	
\nc{\JJ}{\mathbb{J}}	
\nc{\KK}{\mathbb{K}}	
\nc{\LL}{\mathbb{L}}	
\nc{\MM}{\mathbb{M}}	
\nc{\NN}{\mathbb{N}}	
\nc{\OO}{\mathbb{O}}	
\nc{\PP}{\mathbb{P}}	
\nc{\QQ}{\mathbb{Q}}	
\nc{\RR}{\mathbb{R}}	
\rc{\SS}{\mathbb{S}}	
\nc{\TT}{\mathbb{T}}	
\nc{\UU}{\mathbb{U}}	
\nc{\VV}{\mathbb{V}}	
\nc{\WW}{\mathbb{W}}	
\nc{\XX}{\mathbb{X}}	
\nc{\YY}{\mathbb{Y}}	
\nc{\ZZ}{\mathbb{Z}}	

\nc{\bA}{\mathbf{A}}	
\nc{\bB}{\mathbf{B}}	
\nc{\bC}{\mathbf{C}}	
\nc{\bD}{\mathbf{D}}	
\nc{\bE}{\mathbf{E}}	
\nc{\bF}{\mathbf{F}}	
\nc{\bG}{\mathbf{G}}	
\nc{\bH}{\mathbf{H}}	
\nc{\bI}{\mathbf{I}}	
\nc{\bJ}{\mathbf{J}}	
\nc{\bK}{\mathbf{K}}	
\nc{\bL}{\mathbf{L}}	
\nc{\bM}{\mathbf{M}}	
\nc{\bN}{\mathbf{N}}	
\nc{\bO}{\mathbf{O}}	
\nc{\bP}{\mathbf{P}}	
\nc{\bQ}{\mathbf{Q}}	
\nc{\bR}{\mathbf{R}}	
\nc{\bS}{\mathbf{S}}	
\nc{\bT}{\mathbf{T}}	
\nc{\bU}{\mathbf{U}}	
\nc{\bV}{\mathbf{V}}	
\nc{\bW}{\mathbf{W}}	
\nc{\bX}{\mathbf{X}}	
\nc{\bY}{\mathbf{Y}}	
\nc{\bZ}{\mathbf{Z}}	
 
\nc{\calA}{\mathcal{A}}	
\nc{\calB}{\mathcal{B}}	
\nc{\calC}{\mathcal{C}}	
\nc{\calD}{\mathcal{D}}	
\nc{\calE}{\mathcal{E}}	
\nc{\calF}{\mathcal{F}}	
\nc{\calG}{\mathcal{G}}	
\nc{\calH}{\mathcal{H}}	
\nc{\calI}{\mathcal{I}}	
\nc{\calJ}{\mathcal{J}}	
\nc{\calK}{\mathcal{K}}	
\nc{\calL}{\mathcal{L}}	
\nc{\calM}{\mathcal{M}}	
\nc{\calN}{\mathcal{N}}	
\nc{\calO}{\mathcal{O}}	
\nc{\calP}{\mathcal{P}}	
\nc{\calQ}{\mathcal{Q}}	
\nc{\calR}{\mathcal{R}}	
\nc{\calS}{\mathcal{S}}
\nc{\calT}{\mathcal{T}}	
\nc{\cU}{\mathcal{U}}	
\nc{\calV}{\mathcal{V}}	
\nc{\calW}{\mathcal{W}}
\nc{\calX}{\mathcal{X}}	
\nc{\calY}{\mathcal{Y}}	
\nc{\calZ}{\mathcal{Z}}

\nc{\fraka}{\mathfrak{a}}
\nc{\frakb}{\mathfrak{b}}
\nc{\frakc}{\mathfrak{c}}
\nc{\frakd}{\mathfrak{d}}
\nc{\frake}{\mathfrak{e}}
\nc{\frakf}{\mathfrak{f}}
\nc{\frakg}{\mathfrak{g}}
\nc{\frakh}{\mathfrak{h}}
\nc{\fraki}{\mathfrak{i}}
\nc{\frakj}{\mathfrak{j}}
\nc{\frakk}{\mathfrak{k}}
\nc{\frakl}{\mathfrak{l}}
\nc{\frakm}{\mathfrak{m}}
\nc{\frakn}{\mathfrak{n}}
\nc{\frako}{\mathfrak{o}}
\nc{\frakp}{\mathfrak{p}}
\nc{\frakq}{\mathfrak{q}}
\nc{\frakr}{\mathfrak{r}}
\nc{\fraks}{\mathfrak{s}}
\nc{\frakt}{\mathfrak{t}}
\nc{\fraku}{\mathfrak{u}}
\nc{\frakv}{\mathfrak{v}}
\nc{\frakw}{\mathfrak{w}}
\nc{\frakx}{\mathfrak{x}}
\nc{\fraky}{\mathfrak{y}}
\nc{\frakz}{\mathfrak{z}}

\nc{\frakA}{\mathfrak{A}}
\nc{\frakB}{\mathfrak{B}}
\nc{\frakC}{\mathfrak{C}}
\nc{\frakD}{\mathfrak{D}}
\nc{\frakE}{\mathfrak{E}}
\nc{\frakF}{\mathfrak{F}}
\nc{\frakG}{\mathfrak{G}}
\nc{\frakH}{\mathfrak{H}}
\nc{\frakI}{\mathfrak{I}}
\nc{\frakJ}{\mathfrak{J}}
\nc{\frakK}{\mathfrak{K}}
\nc{\frakL}{\mathfrak{L}}
\nc{\frakM}{\mathfrak{M}}
\nc{\frakN}{\mathfrak{N}}
\nc{\frakO}{\mathfrak{O}}
\nc{\frakP}{\mathfrak{P}}
\nc{\frakQ}{\mathfrak{Q}}
\nc{\frakR}{\mathfrak{R}}
\nc{\frakS}{\mathfrak{S}}
\nc{\frakT}{\mathfrak{T}}
\nc{\frakU}{\mathfrak{U}}
\nc{\frakV}{\mathfrak{V}}
\nc{\frakW}{\mathfrak{W}}
\nc{\frakX}{\mathfrak{X}}
\nc{\frakY}{\mathfrak{Y}}
\nc{\frakZ}{\mathfrak{Z}}

\nc{\Lie}{\on{Lie}}
\nc{\GL}{\on{GL}}
\nc{\PGL}{\on{PGL}}
\nc{\SL}{\on{SL}}
\nc{\Sp}{\on{Sp}}
\nc{\GSp}{\on{GSp}}
\nc{\SO}{\on{SO}}
\nc{\Or}{\on{O}}
\nc{\gl}{\on{\mathfrak{gl}}}
\rc{\sl}{\on{\mathfrak{sl}}}

\nc{\Mat}{\on{Mat}}
\nc{\Fun}{\on{Fun}}
\nc{\Aut}{\on{Aut}}
\nc{\End}{\on{End}}
\nc{\Hom}{\on{Hom}}
\nc{\Sym}{\on{Sym}}
\nc{\Span}{\on{span}}
\nc{\Irr}{\on{Irr}}
\nc{\Uch}{\on{Uch}}
\nc{\Spec}{\on{Spec}}

\nc{\Ind}{\on{Ind}}
\nc{\Res}{\on{Res}}
\nc{\stab}{\on{stab}}
\nc{\orb}{\on{orb}}
\rc{\ker}{\on{ker}}
\nc{\im}{\on{im}}
\nc{\tr}{\on{tr}}
\nc{\ord}{\on{ord}}
\nc{\rank}{\on{rank}}
\nc{\Tor}{\on{Tor}}

\nc{\Id}{\on{Id}}
\nc{\Log}{\on{Log}}
\nc{\Exp}{\on{Exp}}
\nc{\Frac}{\on{Frac}}
\nc{\diag}{\on{diag}}
\nc{\D}{\on{D}}

\nc{\St}{\mathrm{St}}
\nc{\triv}{\mathrm{triv}}
\nc{\sgn}{\mathrm{sgn}}
\nc{\reg}{\mathrm{reg}}

\nc{\op}{\mathrm{op}}
\nc{\ad}{\mathrm{ad}}
\rc{\ss}{\mathrm{ss}}
\nc{\HLV}{\mathrm{HLV}}
\nc{\GIT}{\mathrm{GIT}}
\nc{\stack}{\mathrm{stack}}

\allowdisplaybreaks

\begin{document}

\maketitle
\setcounter{tocdepth}{1}

\begin{abstract}Absolutely indecomposable vector bundle and parabolic vector bundles are well-studied via quiver representations. In this paper, we study absolutely indecomposable quasi-parabolic $G$-bundles over $\mathbb{P}^1$ with generic additive character varieties. Furthermore, this gives a geometric interpretation of the multiplicity of the tensor product of irreducible characters of finite reductive groups using  generic additive character varieties. 
\end{abstract}

\tableofcontents

 \section{Introduction}
  
 In this paper, $G$ is a connected (untwisted) reductive group whose derived subgroup is simply connected defined over $k$ (where $k$ is an algebraically closed field with positive characteristic), $F$ the Frobenius map such that $k^F=\Fq$, $T$ a  maximal split torus of $G$, $F$-stable $B$ a Borel subgroup containing $T$, $F$-stable $P$ a parabolic subgroup containing $T$, $Z_G$ the centre of $G$, $W_G$ the Weyl group over $T$, $\fg$ the Lie algebra of $G$, and $\Phi_G$ the root system of $G$. In general, we drop $G$ if it is obvious.
 In addition, we denote the set of (complex) irreducible characters of a finite group $H$ by  $\widehat{H}$.

  \subsection{Absolutely indecomposable quasi-parabolic $G$-bundles}
 A vector bundle over a projective curve $X_{\Fq}$   is called \emph{absolutely indecomposable} if it cannot be decomposed two non-zero vector bundles over $X_{k}$.   Schiffmann showed that the number of absolutely indecomposable rank $n$ vector bundles on $X$ with degree $d$ (when $\mathrm{gcd}(d,n)=1$) is the number of stable Higgs bundles on $X$ with the same rank and degree in \cite{Sch}. Furthermore, Dobrovolska, Ginzburg and Travkin generalised the previous story to the case $\mathrm{gcd}(d,n)\neq 1$. In this case, the number of absolutely indecomposable vector bundles is related to the number of parabolic Higgs bundles, cf. \cite{DGT}. Quiver representations are important in their stories, which are studied well by Kac (e.g. \cite{kac2006root}). Recently, Jakob and Yun counted absolutely indecomposable $G$-bundles over $X_{\Fq}$ over motivic terms in \cite{JY}, and this is related to the moduli stack of parabolic $G$-Higgs bundles. Recall that a principal $G$-bundle $\mathcal{E}$ is called \emph{absolutely indecomposable} when the neutral component of $\mathrm{Aut}(\mathcal{E})$ is unipotent mod the centre.

Furthermore, there is a parabolic version of this story. In \cite{Let16}, Letellier considered the number of quasi-parabolic vector bundles over $\mathbb{P}_{\Fq}^1$ over arbitrary base vector bundle. Recall that quasi-parabolic vector bundle means that this bundle has a parabolic structure  without a weight structure. When the base vector bundle is $\mathcal{O}(a)^n$ (the trivial line bundle $\mathcal{O}$), there is a bijection between the set of (absolutely) indecomposable representations of star-shaped quiver $(\Gamma, v)$ and the set of (absolutely) indecomposable quasi-parabolic vector bundles, cf. \cite[\S2.2.8]{Let16}.
 
 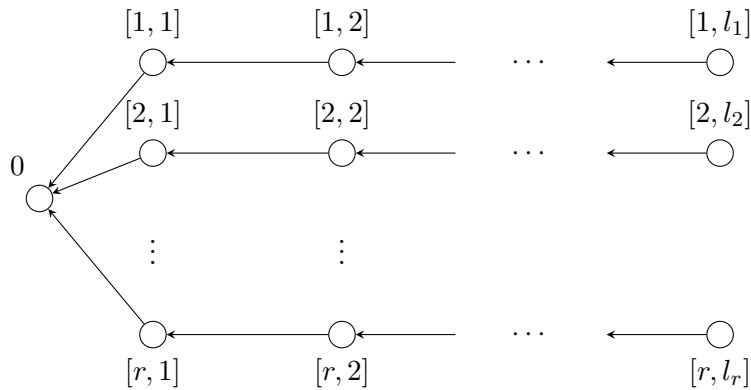
\begin{figure}[H]
 \begin{tikzpicture}[ 
    v/.style={circle, draw, minimum size=3.5mm, inner sep=0pt},
    >=stealth
]
 
    \node[v, label=110:$0$] (0) at (0, 0) {};
 
    \node[v, label=above:{$[1, 1]$}]   (11) at (1.5, 1.8) {};
    \node[v, label=above:{$[1, 2]$}]   (12) at (4.0, 1.8) {};
    \node                              (1d) at (6.5, 1.8) {$\cdots$};
    \node[v, label=above:{$[1, l_1]$}] (1l) at (9.0, 1.8) {};

    \draw[<-] (0)  -- (11);
    \draw[<-] (11) -- (12);
    \draw[<-] (12) -- (5.5, 1.8);
    \draw[<-] (7.5, 1.8) -- (1l);

    \node[v, label=above:{$[2, 1]$}]   (21) at (1.5, 0.6) {};
    \node[v, label=above:{$[2, 2]$}]   (22) at (4.0, 0.6) {};
    \node                              (2d) at (6.5, 0.6) {$\cdots$};
    \node[v, label=above:{$[2, l_2]$}] (2l) at (9.0, 0.6) {};

    \draw[<-] (0)  -- (21);
    \draw[<-] (21) -- (22);
    \draw[<-] (22) -- (5.5, 0.6);
    \draw[<-] (7.5, 0.6) -- (2l);

    \node at (1.5, -0.6) {$\vdots$};
    \node at (4.0, -0.6) {$\vdots$};

    \node[v, label=below:{$[r, 1]$}]   (r1) at (1.5, -1.8) {};
    \node[v, label=below:{$[r, 2]$}]   (r2) at (4.0, -1.8) {};
    \node                              (rd) at (6.5, -1.8) {$\cdots$};
    \node[v, label=below:{$[r, l_r]$}] (rl) at (9.0, -1.8) {};

    \draw[<-] (0)  -- (r1);
    \draw[<-] (r1) -- (r2);
    \draw[<-] (r2) -- (5.5, -1.8);
    \draw[<-] (7.5, -1.8) -- (rl);
    \end{tikzpicture}
    \centering
\caption{Star-shaped quiver}
\label{Star-shaped quiver}
\end{figure}

In this paper,  we study absolutely indecomposable quasi-parabolic $G$-bundles $\mathcal{E}$ over $\mathbb{P}^1:=\mathbb{P}_k^1$ with the (principal) trivial base $G$-bundle. Quasi-parabolic $G$-bundle is constructed by an element $G/P_1\times \ldots \times G/P_\ell$, where $P_1,\ldots , P_\ell$ are parabolic subgroups of $G$ following \cite[Definition (8.3)]{LS}.  Adapting the definition of \cite{JY}, a quasi-parabolic $G$-bundle is called \emph{absolutely indecomposable} if the neutral component of $\mathrm{Aut}(\mathcal{E})$ is unipotent mod the centre.      Furthermore, we consider automorphisms that fix a divisor on $\mathbb{P}^1$  pointwise
and $\ell\geq 3$. Then this implies that every automorphism of $\mathcal{E}$ acts on $\mathbb{P}^1$ by the identity.
The aim of this paper is to consider a relation between   absolutely indecomposable parabolic $G$-bundles and   generic additive character varieties (appearing in \cite{KNWG}) since we are not able to use quiver representations anymore when we consider arbitrary reductive groups.

\subsubsection{Motivation  }\label{sss:moti-future}
When $G=\GL_n$, previous results about absolutely indecomposable vector bundles and parabolic vector bundles (for example, \cite{DGT,kacdt,Let16,Sch}) are deeply related with quiver representations. Furthermore, when we consider a star-shaped quiver with a tuple of indivisible dimension vectors $v:=(v_1, \ldots , v_\ell)$, the size of corresponding quiver variety \[
Q_v:=\{(X_1, \ldots , X_\ell) \in O_1\times \ldots \times O_\ell\,|\, X_1+\cdots + X_\ell=0\}/\!\!/ \mathrm{GL}_n
\]
(where $(O_1, \ldots , O_\ell)$ is a generic semisimple tuple corresponding to $v$)
is the same with the number of absolutely indecomposable representation of this quiver (up to the power of $q$), cf. \cite[Remark 1.3.3]{hausel2011arithmetic}. So we can see that there is a relation between absolutely indecomposable vector bundles and a quiver variety. However, if we consider arbitrary reductive groups,   we could not use quiver representations, and so we focus on an additive character variety.

From now on, let us consider the case   $v=((n,n-1, \ldots , 1),\ldots ,(n,n-1, \ldots , 1))$, and each $O_i$ is a regular semisimple adjoint orbit in $\mathfrak{gl}_n$. Then we can consider the following $\tilde{Q}_v:$
\[
 \left\{ (X_1, \ldots , X_\ell, g_1B, \ldots, g_\ell B)\in \prod_{i=1}^\ell O_i \times  (\mathrm{GL}_n/B)^\ell\,|\, \sum_{i=1}^\ell X_i=0,\ \mathrm{Ad}_{g_i}(X_i)\in O_i \cap \ft\text{ for all }i\right\}/\!\!/\mathrm{GL}_n.
\]
Note that we have an isomorphism to $Q_v$ by considering the projection map to the first $\ell$-elements in $\tilde{Q}_v$. In addition, the structure of $\tilde{Q}_v$ gives an idea to study a relation between an additive character variety and absolutely indecomposable $p\in (G/B)^\ell$ (i.e., an absolutely indecomposable quasi-parabolic $G$-bundle over the trivial base $G$-bundle).

 \begin{rem}
 Our result is related to a type of an additive Deligne-Simpson problem such that finding $(p_1, \ldots , p_\ell)\in  (P_1, \ldots , P_\ell)$ and $(s_1 , \ldots , s_\ell)\in \ft^\ell$ for a given $(g_1, \ldots , g_\ell)\in G^\ell$. Under the generic condition on $(s_1, \ldots , s_\ell)$, the tuple $(\Ad_{g_1p_1}(s_1), \ldots , \Ad_{g_\ell p_\ell}(s_\ell))$ is irreducible. Recall that an additive Deligne-Simpson problem is about the existence of a solution $(g_1, \ldots , g_\ell)\in G^\ell$ of the equation $\Ad_{g_1}(X_1)+\cdots +\Ad_{g_\ell}(X_\ell)$ for a given $(X_1, \ldots , X_\ell)\in \fg^\ell$ such that $(\Ad_{g_1}(X_1),\ldots , \Ad_{g_\ell}(X_\ell))$ is irreducible.   
 \end{rem}

 \subsubsection{Result}
 The main starting object of this paper is quasi-parabolic $G$-bundles with the trivial base $G$-bundle corresponding to elements in $(G/B)^\ell$. Note that when $\mathcal{E}$ corresponds to an element $p\in (G/B)^\ell $, we have $\mathrm{Aut}(\mathcal{E})= \mathrm{Stab}_G(p)$, where $G$ acts on $p$ by the diagonal left multiplication. This is because an automorphism of $\mathcal{E}$ is determined by a regular function from $\mathbb{P}^1$ to $G$, and such regular function is constant (due to  the property of a projective variety).
 From this observation, we consider elements in $(G/B)^\ell$ instead of   quasi-parabolic $G$-bundles over $\mathbb{P}^1$ with the trivial base $G$-bundle. Obviously, an element  $p$ in $(G/B)^\ell$ is called \emph{absolutely indecomposable} when $\mathrm{Stab}_G(p)^\circ/Z_G^\circ$ is unipotent. This is equivalent to that the pseudo-Levi subgroup \(C_G(t)\) is isolated  for every \(t \in \mathrm{Stab}_G(p)^{ss}$, where $H^{ss}$ is the set of semisimple elements in $H$.  
 One of our main result is a relation between  absolutely indecomposable quasi-parabolic $G$-bundles and algebraic equations as follows.
 
 \begin{thm}[Theorem \ref{conj:between2and3} and \ref{thm:between2and3}] \label{thm:main-abs}
 An element $(g_1B,\ldots , g_\ell B)\in (G/B)^\ell$ is absolutely indecomposable for given $g_1, \ldots , g_\ell\in G$ if and only if there exists $(b_1, \ldots , b_\ell)\in B^\ell$ such that 
\begin{equation}\label{eq:thm1}
 	\mathrm{Ad}_{g_1 b_1}(s_1) + \cdots  + \mathrm{Ad}_{g_{\ell} b_{\ell}}(s_{\ell})
		= 0
 \end{equation}
 for a tuple of generic regular semisimple adjoint orbits  $(\mathrm{Ad}_G(s_1), \ldots , \mathrm{Ad}_G(s_\ell))$.
 \end{thm}
 For the definition of generic, please see \S\ref{s:genericdefe}. 
  We prove this theorem by showing each direction in   Theorem \ref{conj:between2and3} and Theorem \ref{thm:between2and3}. 
 
 \bigskip
 We can extend this theorem to $G/P_1\times\ldots \times G/P_\ell$, where every $P_i=L_i U_i$ is a parabolic subgroup satisfying a condition that there exists a generic tuple $(s_1, \ldots , s_\ell)\in \ft^{\ell}$ such that $C_G(s_i)=L_i$ for all $i$. 
 In addition, we call such tuple $(s_1, \ldots , s_\ell)$ as a tuple of generic semisimple $(P_1, \ldots , P_\ell)$-type. Note that it is easy to see that there is a tuple of generic semisimple $(B,\ldots , B)$ type (when the size of tuple is at least $3$), cf. \cite[\S3.1.6]{KNWG}.

 The proof of Theorem \ref{thm:main-abs}   gives a simple reduction to Theorem \ref{thm:generalisation}.
  With this reason, we start from the Borel case, and then extend to the parabolic cases.
 \begin{thm}[Theorem \ref{thm:extension}]\label{thm:main-abs2}\label{thm:generalisation} Let us assume that there exists a tuple of generic semisimple $ (P_1, \ldots , P_\ell)$-type adjoint orbits $(\Ad_G(s_1),\ldots , \Ad_G(s_\ell))$.
 Then an element $(g_1P_1, \ldots , g_\ell P_\ell)\in G/P_1\times \ldots \times G/P_\ell$ is absolutely indecomposable for given $g_1, \ldots , g_\ell\in G$ if and only if there exists $(p_1, \ldots , p_\ell) \in (P_1, \ldots , P_\ell)$ such that 
   \[
   \Ad_{g_1p_1}(s_1)+\cdots+   \Ad_{g_\ell p_\ell}(s_\ell)=0
   \]
   for a tuple of generic semisimple $ (P_1, \ldots , P_\ell)$-type adjoint orbits $(\Ad_G(s_1),\ldots , \Ad_G(s_\ell))$.
   \end{thm}
This provides an algebraic method to determine whether a given quasi-parabolic $G$-bundle is absolutely indecomposable.
    
 \begin{rem}
Note that such results are independent on the choice of $(s_1, \ldots , s_\ell)$ whenever this tuple satisfies the generic condition and $ (P_1, \ldots , P_\ell)$-type. 
\end{rem}
 
The next result is about the number of absolutely indecomposable quasi-parabolic $G$-bundles up to isomorphism via the multiplicity.
 
 \begin{thm}\label{thm:main-abs-multiplicity} Let us denote the number of absolutely indecomposable quasi-parabolic $G$-bundle in $(G^F/P_1^F,\ldots , G^F/P_{\ell}^F)$ (up to $G^F$-action) by $A_{G^F,P_1\times\cdots\times P_\ell}$ with $P_i=L_iU_i$.
Suppose that there exists  a tuple of  generic linear characters $(\theta_1,\ldots ,\theta_\ell) $  of $L_1^F\times \cdots \times L_\ell^F$
in the sense of Definition~\ref{def:generic-char}. Then
we have
\[
A_{G^F,P_1\times\cdots\times P_\ell} = \left\langle\bigotimes_{i=1}^{\ell}R_{L_i}^{G}(\theta_i),1
       \right\rangle_{G^F}.
\]
\end{thm}
The next our result is that to compute $A_{G^F,P_1\times \cdots \times P_\ell}$ when $P_i=B $ for all $i$. The reason is that in this case, we can check the existence of such generic characters, and furthermore we can find a nice relation with generic additive character varieties.

\begin{rem} For these $F$-stable parabolic subgroups, Harish--Chandra induction is the
ordinary induction
$ R_{L_i}^{G}(\theta_i)
 :=\Ind_{P_i^F}^{G^F}(\mathrm{Infl}_{L_i^F}^{P_i^F}(\theta_i)),$ cf. \cite[Proposition 3.3.3]{geck2020character}.
\end{rem}

 \subsection{Multiplicity of the tensor product of irreducible characters}
Let $\chi_1, \ldots , \chi_{\ell}$ be irreducible characters of \(G^F\) over $\mathbb{C}$. When we consider irreducible characters of a finite group, a natural question is to determine the value of the inner product
 \begin{equation}\label{eq:def-multi}
 \langle \chi_1 \otimes \ldots  \otimes \chi_{\ell}, 1\rangle.
 \end{equation}
 This is related to the ring structure of the space of finite dimensional characters of $G^F$.
 To the best of our knowledge, there are relatively few known results concerning this inner product of irreducible characters of finite reductive groups. One notable result, due to Heide, Saxl, Tiep, and Zalesski~\cite[Theorem~1.2]{heide2013conjugacy}, states that
$ \langle St \otimes St \otimes \chi,\,1\rangle \neq 0,$
 for almost all simple Lie groups \( G \) and any irreducible character $\chi$ of $G^F$, where \(St\) denotes the Steinberg character of $G^F$.
 
  When we consider \( G = \mathrm{GL}_n \),  results of the multiplicity  have been extensively studied by Letellier in~\cite{letellier2013tensor, letellier2013quiver}. In particular, Letellier computed this inner product explicitly   involving quiver varieties.
 Motivated by these work, in this paper, we have identified a new connection between  additive character varieties (introduced in \cite{KNWG}) and the multiplicity of irreducible characters of finite reductive groups.

	
	\begin{assu}\label{ass}
	In this paper, we assume that the characteristic $p$ of $k$ is very good, and the size $q$ is sufficiently large. This  guarantees the existence of a tuple of generic linear characters $(\theta_1,\ldots , \theta_\ell)$ of $(T^F)^\ell$ by assuming $\theta_1,\ldots , \theta_\ell$ are all in general position,\footnote{Here, in general position means that $g.\theta \neq \theta$ for all $g\in N_G(T)^F/T^F$ (except $g\in T^F$), cf. \cite[Corollary 2.2.9]{geck2020character}.}   cf. \cite[\S8.4]{carter1985finite}.  
\end{assu}
	 
Under the in general position assumption on  $\theta_i$  we have the Deligne-Lusztig character $\chi_{\theta_i}=R_T^G(\theta_i)$. Note that $R_T^G(\theta_i)=\mathrm{Ind}_{T^F}^{G^F}(\theta_i)$ since $T$ is maximally split, and it is well-known that this is an irreducible character of $G^F$ from \cite[Corollary 2.2.9]{geck2020character}.
 Now, let us introduce an additive character variety from \cite[\S1.3]{KNWG}.
  \begin{defe}\label{defe:genericadditivechara}
  	Let us consider a tuple $\mathcal{O}\colonequals (O_1, \ldots, O_{\ell})$ of adjoint orbits in $\mathfrak{g}$ under $G$-action. Then the \emph{additive character variety corresponding to $\mathcal{O}$} is the following:
  	\[
  	\mathcal{X}_{\fg, \mathcal{O} } \colonequals \left\{(X_1, Y_1, \ldots , X_g,Y_g,Z_1, \ldots , Z_{\ell} ) \in \fg^{2g}\times \prod_{i=1}^{\ell} O_i  \,\middle\vert\, \sum_{i=1}^g[X_i,Y_i]+\sum_{j=1}^{\ell} Z_j=0\right\}/\!\!/G,
  	\]
  	where $/\!\!/$ means the GIT quotient.
  \end{defe}
Note that we always consider generic regular semisimple adjoints orbits in this paper, so we drop to denote $\mathcal{O}$, i.e., we use $\mathcal{X}_\fg$.
  
  \subsubsection{Result} 
  Let us assume that $2g+\ell\geq 3$ for $g\in \mathbb{N}\cup \{0\}$ and $\ell \in \mathbb{N}$, and define a function $\Lambda \,:\, G\rightarrow \mathbb{C}$   by $x\rightarrow q^{g \dim(C_G(x))}$, where $C_G(x)$ is the centraliser subgroup of $x$ in $G$.
 Then with the generic condition (Definition \ref{defe:genericoriginal}) on $(O_1, \ldots , O_{\ell})$, we have
a following geometric interpretation of the multiplicity     $    \langle \Lambda \otimes \chi_{\theta_1}\otimes \ldots \otimes \chi_{\theta_{\ell}},1 \rangle$ as follows.

\begin{thm}[Theorem \ref{thm:main}] \label{thm:maininintro}
    Let us assume that $(\theta_1, \ldots,\theta_{\ell})$ is a tuple of generic linear characters of $(T^F)^\ell$ as in Assumption. Then the  multiplicity 
    $    \langle \Lambda \otimes \chi_{\theta_1}\otimes \ldots \otimes \chi_{\theta_{\ell}},1 \rangle$)
 is given by sum of the number of points over finite fields of the generic additive character varieties associated with isolated pseudo-Levi subgroups of $G$. In addition, if $g=0$, we also get the formula for $A_{G^F,B\times \cdots \times B}$. 
\end{thm} 
The explicit formula is given in Equation \eqref{eq:mainresult}, and furthermore, we have the following conjecture.
\begin{conj}
Every coefficient of the multiplicity $      \langle \Lambda \otimes \chi_{\theta_1}\otimes \ldots \otimes \chi_{\theta_{\ell}},1 \rangle$ is non-negative coefficients (over $q$). 
\end{conj}
When $G=\GL_n$, this is true from the purity of the quiver variety. From Equation \eqref{eq:mainresult}, we also guess that this is true if every generic additive character variety is pure, cf. \cite[Conjecture 9]{KNWG}.

 \subsection{Structure of the paper}
 In this paper, we give basic materials in \S\ref{s:prelim}. We prove Theorem \ref{thm:main-abs} in \S\ref{s:frombtoa}--\ref{ss:converse-thm-1} and its extension to the parabolic case in \S\ref{ss:generisation}. We consider the number of absolutely indecomposable elements in \S\ref{s:analogueof5.7}. We give the multiplicity result in \S\ref{s:regularsemisimplechar}.  Furthermore, we illustrate an example in Appendix \ref{s:example} about Theorem \ref{thm:main-abs}  over $\mathrm{SO}_5$.

 	\section{Preliminary} \label{s:prelim} 
 	In this section, we recall basic facts and notations about root systems, reductive groups and a generic regular semisimple tuple, which  illustrated in \cite[\S2]{KNWG}. For convenience, we recall some materials here briefly.    	
 	
 	\subsection{Root systems} \label{ss:root-system}
 	Let $\Phi$ be a (reduced crystallographic) root system in a finite dimensional Euclidean vector space $(V, (.\ ,.))$.  We assume that $\Phi$ spans $V$ and fix a base $\Delta$ of $\Phi$. Then $\Delta$ is a basis of the vector space $V$. Let $W$ be the Weyl group generated by the simple reflections $s_\alpha$, $\alpha \in \Delta$. 
 	
 	\subsubsection{Subsystems} A root subsystem of $\Phi$ is a subset $\Psi\subseteq \Phi$ which is itself a root system. Equivalently, it satisfies $s_\alpha(\beta)\in \Psi$ for all $\alpha, \beta\in \Psi$. A subsystem $\Psi\subseteq \Phi$ is called \emph{closed} if 
 $	\alpha, \beta\in \Psi \quad \textrm{and}\quad  \alpha+\beta\in \Phi$ implies $ \alpha+\beta\in \Psi. $
 	Given a subset $S\subseteq \Phi$,  let $\langle S\rangle_{\mathbb{Z}}$ denote the subgroup of $V$ generated by $S$.  The subsystem of $\Phi$ generated by $S$ is defined by 
 	\[
 	\langle S \rangle:=\langle S\rangle_{\mathbb{Z}} \cap \Phi. 
 	\]
 	It is easy to check that $ \langle S \rangle$ is a closed subsystem of $\Phi$.

 	A \emph{Levi subsystem}  of $\Phi$ is a subsystem of the form $\Phi \cap U$ where $U\subseteq V$ is a subspace. 
 	A Levi subsystem of $\Phi$ is closed in $\Phi$. Note that every Levi subsystem of $\Phi$ is of the form $w\cdot \langle S\rangle$, where $S$ is a subset of $\Delta$ and $w\in W$. In particular, if $w=1$, then $\langle S \rangle$ is called a standard Levi subsystem. 
 	A subset $S\subseteq \Phi$ is called \emph{isolated in $\Phi$} if $|S|=|\Delta|$, and this implies that $\langle S \rangle$ is not contained in a proper Levi subsystem of $\Phi$.

 	\subsubsection{Pseudo-Levi subsystems} Assume $\Phi$ is irreducible. Let $\theta$ be the highest root of $\Phi$, and $\widetilde \Delta:=\Delta \sqcup \{-\theta\}$. A pseudo-Levi subsystem is a subsystem of $\Phi$ of the form $w\cdot \langle S\rangle $, where $w\in W$ and $S$ is a subset of $\widetilde \Delta$. If $w=1$, then $\langle S\rangle$ is called a standard pseudo-Levi subsystem. 
	This story can be extended to non-irreducible $\Phi$ easily.

 	\subsubsection{Isolated pseudo-Levi subsystem} 
 	Let $\Psi$ be a pseudo-Levi subsystem of an irreducible root system $\Phi$.  Then we have the following equivalence: $\Psi$ is isolated in $\Phi$ if and only if 
 		  $\Psi=w\cdot \langle S\rangle $ where $S\subset \tilde{\Delta}$ has the same size as $\Delta$.  
 	In other words, up to $W$-conjugation, isolated pseudo-Levis are those which are obtained by removing just a single element of $\widetilde{\Delta}$ (equivalently, a single node from the extended Dynkin diagram).

 	\subsection{Reductive groups} \label{s:Reductive}

 	\subsubsection{Levi  and pseudo-Levi subgroups}\label{sss:levisubgrous} A Levi subgroup of $G$ is defined to be the centraliser $L=C_G(S)$ of a torus $S\subseteq G$. If $S\subseteq T$, then $L$ is said to be a standard Levi subgroup. In this case, the root system of $L$ is a Levi subsystem of $\Phi$.
 	If $t\in \fg$ is a semisimple element, then the connected centraliser $C_{G}(t) $ is a Levi subgroup of $G$. Moreover, every Levi subgroup of $G$ arises in this manner. Note that if the characteristic of $k$ is very good for $G$, then $C_G(t)$ is connected \cite[Theorem 3.14]{Steinberg75}.

	Let $t\in G$ be a semisimple element. Recall that if the derived subgroup $[G,G]$ of $G$ is simply connected, then the centraliser subgroup $C_G(t)$ of a semisimple element $t\in G$ is connected. The connected centraliser $C_G(t) $ is called a \emph{pseudo-Levi subgroup} of $G$, and if $t\in T$, then $C_G(t) $ is called a standard pseudo-Levi subgroup of $G$. We call a pseudo-Levi subgroup is isolated if its root subsystem is isolated in $\Phi$. Note that an isolated pseudo-Levi subgroup is not in proper parabolic subgroup. For a pseudo-Levi subgroup $E$ of $G$, we denote the root subsystem of $E$ in $\Phi$ by $\Phi_E$.
 	
	\subsubsection{Revisit of pseudo-Levi subsystem}
	 If $t\in T$, then the root system of $C_G(t) $ is given by 
${\Phi}_t:=\{{\alpha} \in {\Phi} \, |\, {\alpha}(t)=1\}.$
This is a pseudo-Levi subsystem of $\Phi$. Every pseudo-Levi subsystem arises in this manner \cite{Deriziotis}. From this observation, we have the following useful lemma.
\begin{lem}\label{lem:levi-E-pseudo-levi-G}
 
For any Levi subsystem $\Psi$ of an isolated pseudo-Levi subsystem $\Phi_E$, there is $t\in T$ such that $\Phi_t  = \Psi$.
\end{lem}
 	 \begin{proof}It is sufficient to prove the case when $G$ is simple. 
	  Let $S\subset \Delta(E)$ such that $\Psi=w\cdot \langle S\rangle  $. Our goal is to find $t\in T$ such that $\alpha(t)=1$ for $\alpha \in S$ and $\beta(t)\neq1 $ for $\beta\in \tilde{\Delta}\setminus S$. This is possible since $|\tilde{\Delta}\setminus S|\geq 2$.
Therefore, we are done.
	 \end{proof}

 	 \subsection{Generic semisimple tuple}\label{s:genericdefe} Let us define the notion of a \emph{generic semisimple tuple} in the Lie algebra $\fg$ of $G$. This serves as an additive analogue of the concept introduced in \cite[\S3]{KNWG}, where the authors primarily focused on the group-theoretic setting. For convenience, we present here the corresponding additive version.

 	 	\begin{defe}\label{defe:genericoriginal}
 	 		The tuple $\mathcal{O}=(O_1,\ldots, O_{\ell})$ of semisimple adjoint orbits in $\fg$ is called \emph{generic} if $\sum_{i=1}^\ell x_i\in [\fg,\fg]$ for any $x_i\in O_i$, and whenever there exists a proper parabolic subalgebra $\mathfrak{p}\subset \fg$ and $x_i\in \mathfrak{p}\cap O_i$,  for all $i=1,\ldots,\ell$, then  $\sum_{i=1}^{\ell} x_i\notin [\mathfrak{p},\mathfrak{p}]$. 
 	 	\end{defe}

 	 	\subsubsection{Standard parabolic subalgebras} It is enough to check the generic condition for \emph{standard} (proper) parabolic  subalgebras. Here, standard parabolic subalgebras are parabolic subalgebras containing a chosen Borel subalgebras $\mathfrak{b}\subset \fg$. Indeed, suppose $\fp$ is an arbitrary parabolic and choose $g\in G$ such that $\Ad_g(\fp)$ is standard. Then $x_i\in\fp\cap O_i$ if and only if $\Ad_g(x_i)\in \Ad_g(\fp)\cap O_i$, and  $\sum_{i=1}^{\ell} x_i\in [\fp,\fp]$ if and only if $\Ad_g(\sum_{i=1}^{\ell} x_i)= \sum_{i=1}^{\ell} \Ad_g(x_i)\in \Ad_g([\fp,\fp])=[\Ad_g(\fp), \Ad_g(\fp)]$.

 	 	\subsubsection{Levi subalgebras instead of parabolic subalgebras}\label{sss:algebra-parabolic-generic}   We can check genericity using  Levi subalgebras instead of parabolics for semisimple adjoint orbits. In other words, one can check that a tuple of semisimple adjoint orbits  $O=(O_1,\ldots,O_{\ell})$ is generic if and only if whenever there exists a proper standard Levi subalgebra $\fl\subset \fg$ and $s_i\in \fl\cap O_i$,  for $i=1,\ldots,\ell$, then  $\sum_{i=1}^{\ell} s_i\notin [\fl,\fl]$. This comes from the fact that if $l\in \fl$, then $l\in [\fl,\fl]$ if and only if $l\in [\fp,\fp]$.

 	 	\begin{defe}\label{defe:genericelement} A tuple $ (s_1,\ldots,s_{\ell})\in \ft^{\ell}$ is called to be \emph{generic} if $\sum_{i=1}^\ell s_i \in [\fg,\fg]$, and for every $w=(w_1,\ldots, w_{\ell}) \in W^{\ell}$ and every proper standard Levi subalgebra $\mathfrak{l}\subset \fg$, we have $\sum_{i=1}^{\ell} w_i\cdot s_i \notin [\mathfrak{l},\mathfrak{l}]$. 
 	 	\end{defe} 
 	 	
 	 	Let $O_i:= G\cdot s_i$ be the adjoint orbit corresponding to $s_i$. Then we can give a useful method to check whether a given tuple $(O_1, \ldots , O_{\ell})$ is generic or not.
 	 	
 	 	\begin{lem}\label{lem:equivalence}  The element $ (s_1,\ldots,s_{\ell})\in \ft^{\ell}$ is generic if and only if the tuple of adjoint orbits $ (O_1,\ldots,O_{\ell})$ is generic. 
 	 	\end{lem}

 	 	\begin{proof} It is clear that if $(O_1,\ldots,O_{\ell})$ is generic then so is $(s_1,\ldots,s_{\ell})$. 
 	 		For the converse, let $ (x_1,\ldots,x_{\ell}) \in \prod_{i=1}^\ell O_i $ and $\fl\subset \fg$ a proper standard Levi subalgebra containing $x_1,\ldots,x_{\ell}$. To prove this, it is enough to show that $\sum_{i=1}^{\ell} x_i\notin [\fl,\fl]$. Note that $x_i=\Ad_{l_i}(t_i)$ for some $l_i\in L$ and $t_i\in \ft$. Since $t_i$ and $s_i$ are elements of $\ft$ which are $G$-adjugate; thus, they must actually be adjugate under $W$. This means that there exists $w_i\in W$ such that $t_i=w_i\cdot s_i$. To prove our goal, let us assume that $\sum_{i=1}^{\ell} x_i\in [\fl,\fl]$. Then with the fact that $\Ad_g(x)-x\in [\fg,\fg]$ for all $g\in G$ and $x\in \fg$ (cf. \cite[Chapter I, 3.16 (a)]{borel}), we have that 
			\[
\begin{split}	\sum_{i=1}^{\ell} x_i-\sum_{i=1}^{\ell} (\Ad_{l_i} (t_i)-t_i)\in [\fl,\fl]
\iff  \sum_{i=1}^{\ell} t_i= \sum_{i=1}^{\ell} w_i \cdot s_i\in [\fl,\fl].
\end{split}		\]
This contradicts the assumption that the tuple $(s_1, \ldots , s_{\ell} ) \in \ft^{\ell}$ is generic, so we conclude that $\sum_{i=1}^{\ell} x_i \notin [\fl,\fl]$. Therefore, the tuple $(O_1, \ldots , O_{\ell})$ is generic. 	 	\end{proof}

 	  When $G=\GL_n$, note that we can check the genericity by considering subsums of eigenvalues as in \cite{hausel2011arithmetic}.

\subsection{Generic characters}\label{ss:generic-chara}

Let us consider parabolic subgroups $P_1,\ldots , P_\ell$ with $P_i=L_iU_i$.
Let $\theta_i:L_i^F\rightarrow \mathbb{C}^\times$ be linear characters and inflate them to
$\widetilde\theta_i=\theta_i\circ\pi_i$ on $P_i^F$.
If $J$ is an $F$-stable  pseudo-Levi and
$ (g_1,\ldots,g_\ell)\in(G^F)^\ell$ satisfies
\begin{equation}\label{eq:compatible}
 Z_J\subseteq\bigcap_{i=1}^\ell g_iP_ig_i^{-1},
\end{equation}
then we define the linear character corresponding to $(\theta_1,\ldots , \theta_\ell)$ as follows:
\begin{equation}\label{eq:central-character}
 \Theta_{J,(g_1,\ldots,g_\ell)}:=  \Theta_{J,(g_1,\ldots,g_\ell)}^{(\theta_1,\ldots,\theta_\ell)}:Z_J^F\longrightarrow\mathbb{C}^\times
 \quad\text{given by} \quad
 \Theta_{J,(g_1,\ldots,g_\ell)}(z)=
 \prod_{i=1}^{\ell}\widetilde\theta_i(g_i^{-1}zg_i).
\end{equation}

\begin{defe}\label{def:generic-char}
The tuple $ (\theta_1,\ldots,\theta_\ell)$ of linear characters of $L_1^F\times \cdots \times L_\ell^F$ is called
\emph{generic}  if the following conditions hold:
\begin{enumerate} 
\item[(i)] Each $\theta_i$ is trivial on every unipotent element of $L_i^F$.
\item[(ii)] For every proper $F$-stable  Levi subgroup $J\subsetneq G$
      and every $(g_1,\ldots,g_\ell)\in(G^F)^\ell$ satisfying Equation \eqref{eq:compatible},
      one has $  \Theta_{J,(g_1,\ldots,g_\ell)}\neq1$.
\item[(iii)] For every $F$-stable isolated pseudo-Levi subgroup $J\subseteq G$
      and every $(g_1,\ldots,g_\ell)\in(G^F)^\ell$ satisfying Equation \eqref{eq:compatible},
      one has $ \Theta_{J,(g_1,\ldots,g_\ell)}=1$.
\end{enumerate}
\end{defe}

\section{Absolutely indecomposable quasi-parabolic $G$-bundles}\label{s:stabiliser-problem}
Let us consider a relation between an absolutely indecomposable quasi-parabolic $G$-bundles and Lie algebraic equations.

\subsection{The forward direction of Theorem \ref{thm:main-abs}}\label{s:frombtoa}

\begin{thm}\label{conj:between2and3}
Let \((g_1B, \ldots, g_{\ell}B) \in (G/B)^{\ell}\) be absolutely indecomposable for given $g_1,\ldots , g_\ell\in G$.  	
	Then there exist \((b_1, \ldots , b_{\ell}) \in B^{\ell}\) such that
		\begin{equation}\label{eq:solution}
		\mathrm{Ad}_{g_1 b_1}(s_1) + \cdots  + \mathrm{Ad}_{g_{\ell} b_{\ell}}(s_{\ell})
		= 0
		\end{equation}
			for some regular elements $s_1,\ldots, s_{\ell} \in \ft$ such that $(\mathrm{Ad}_{G}(s_1), \ldots ,\mathrm{Ad}_{G}(s_{\ell})))$ is generic.
\end{thm}
 \begin{proof}
From the definition of generic, we have that $s_1+\cdots +s_{\ell}\in [\fg,\fg]$. In addition, we can assume that
$g_1, \ldots , g_{\ell},b_1, \ldots , b_{\ell}\in [G,G]$ since $G=Z_G[G,G]$ and $\Ad_{z}(X)=X$ for $z\in Z_G$ and $X\in \fg$. Therefore, let us assume that $G$ is semisimple.

It is sufficient to prove the following. Let 
$$X:=\{(X_1, \ldots , X_{\ell})\in \Ad_{g_1}(\mathfrak{b})\times \ldots \times \Ad_{g_{\ell}}(\mathfrak{b}) \,|\, X_1+\cdots +X_{\ell}=0\}.$$ Then the following map $\phi$ is surjective:
\[
\phi\,:\, X\rightarrow \ft^\ell\quad \text{given by }\phi((X_1, \ldots , X_{\ell}))=(\Ad_{g_1^{-1}}(\tilde{t}_1), \ldots , \Ad_{g_{\ell}^{-1}}(\tilde{t}_{\ell})) ,
\]
where $X_i:=\tilde{t}_i+n_i\in \Ad_{g_i}(\ft+\fn)$.

If this map is surjective, then using the fact that the tuple of regular generic elements in $\ft^n$ is dense and open, we can find  a generic regular tuple $(s_1, \ldots , s_\ell)\in \ft^\ell$ and $(b_1, \ldots , b_\ell)\in B^\ell$ such that $\Ad_{g_1b_1}(s_1)+\cdots +\Ad_{g_1b_\ell}(s_\ell)=0$ by taking $b_i\in B$ such that $\Ad_{b_i}(s_i)=s_i+n_i$ for $(s_1+n_1, \ldots , s_\ell+n_\ell)\in \phi^{-1}((s_1, \ldots , s_\ell)).$
 (Recall that $t+n$ is regular semisimple element if $t\in \ft$ is regular  and $n\in \fn$.) The claim that $\phi$ is surjective will be proved the following Lemma \ref{lem:dominant}, and so we are done.
 \end{proof}
    
 Note that our proof implies that there exists $b_1,\ldots , b_\ell\in B$ for each regular semisimple tuple $(s_1, \ldots , s_\ell)$ satisfying Equation \eqref{eq:solution}. However, the generic condition is important to the converse direction,   so we focus on the generic condition in Theorem \ref{conj:between2and3}.
    
    \begin{lem}\label{lem:dominant}
    Let $G$ be a connected semisimple simply connected group, its Borel subgroup $B$, and their Lie algebras $\fg$ and $\fb=\ft\oplus \fn$. Let us assume that every semisimple element $t\in \mathrm{Stab}_G(g_1B,g_2B,\ldots , g_{\ell}B)^{ss}$ is isolated. Then the   map $  \phi\,:\, X\rightarrow \ft^{\ell} $  is surjective
    
    \end{lem}
    \begin{proof}
To show the surjectivity, let us show that the dual map $\phi^*:(\ft^*)^{\ell}\rightarrow X^*$ is injective. 
Let us take $(\lambda_1, \lambda_2,\ldots , \lambda_{\ell})=:\lambda \in \mathrm{ker}(\phi^*)$, i.e., $\lambda((t_1,t_2,\ldots , t_{\ell}))=\lambda_1(t_1)+\lambda_2(t_2)+\cdots +\lambda_{\ell}(t_{\ell}).$

 Let $\tilde{\phi}=(\tilde{\phi}_1,\tilde{\phi}_2,\ldots , \tilde{\phi}_{\ell})$ be a natural extension of $\phi$ (of domain)  from $\Ad_{g_1}(\fb)\times \Ad_{g_2}(\fb)\times\ldots \times  \Ad_{g_{\ell}}(\fb) $ to $\ft^{\ell}$, and  consider the following diagram \[
\begin{tikzcd}
\Ad_{g_1}(\fb) \times \ldots \times \Ad_{g_{\ell}}(\fb) \arrow[rr, "g"] \arrow[dr, "\lambda\circ \tilde{\phi}"] & & \fg \arrow[dl, "h"] \\
& k &
\end{tikzcd}
\]
(recall that $k$ is a field). Furthermore, let us take the map $g$ by the summation map, i.e., $$g((x_1,x_2,\ldots , x_{\ell}))=x_1+x_2+\cdots +x_{\ell}.$$ Then to make a commuting diagram, let us construct $h$ as follows: $$h(g((x_1,x_2,\ldots ,x_{\ell})))=\lambda\circ \tilde{\phi}((x_1, x_2,\ldots , x_{\ell})).$$ Note that this is well-defined since $\mathrm{ker}(g)\subset \mathrm{ker}(\lambda \circ \tilde{\phi})$. (This inclusion comes from the construction of $\lambda$, and the well-definedness of $h$ holds since every map is linear.)
Note that we can check that 
\[
\lambda\circ\tilde{\phi}|_{\Ad_{g_i}(\fb)}=\lambda_i\circ\tilde{\phi}_i.
\]

Now, let us make $h\in \fg^*$ by extending $h\,:\, \mathrm{Im}(g)\rightarrow k$. Then using the isomorphism $\fg$ and $\fg^*$ via a $G$-invariant non-degenerate bilinear form $\kappa$, we have
\[
h(x)=\kappa(x,y)\quad \text{for some }y \in \fg.
\]
Our focus is to figure out what is $y$. To do this, let us consider the following   
\begin{equation}\label{eq:lambdaandkappa}
\begin{split}
\lambda\circ\tilde{\phi}((\Ad_{g_1}(n_1),0,\ldots,0))&=h(g((\Ad_{g_1}(n_1),0,\ldots , 0)))\\
&=h(\Ad_{g_1}(n_1))\\
&=\kappa(\Ad_{g_1}(n_1),y)
\end{split}
\end{equation} for any nilpotent element $n_1\in \fn$.
However, we have $$\lambda\circ\tilde{\phi}((\Ad_{g_1}(n_1),0,\ldots, 0))=\lambda((0,0,\ldots, 0))=0$$ since $\tilde{\phi}_1(\Ad_{g_1}(n_1))=0,$ and so this gives that
\[
\kappa(\Ad_{g_1}(\fn),y)=0.
\]
This implies that $y\in \Ad_{g_1}(\fb)$ since the orthogonal complement of a nilpotent radical of $\fb$ is $\fb$. With the same steps from $g_2$ to $g_{\ell}$, we can check that
\[
y\in \Ad_{g_1}(\fb)\cap \Ad_{g_2}(\fb)\cap\ldots \cap \Ad_{g_{\ell}}(\fb).
\]
Now, use the condition   that every semisimple element $t$ in $\mathrm{Stab}_G(g_1B,g_2B,\ldots , g_{\ell}B)$ is isolated. This means that when $y\in \Ad_{g_1}(\fb)\cap \Ad_{g_2}(\fb)\cap\ldots \cap \Ad_{g_{\ell}}(\fb)$, its semisimple part is in the centre of $\fg$. Then since $G$ is semisimple, its semisimple part is $0$, and so $y$ is a nilpotent element in the intersection. 

Let us finish the proof by showing that each $\lambda_i$ is zero function. 
For any $t\in \Ad_{g_i}(\ft),$ we have
\[
\lambda_i\circ \tilde{\phi}_i(t)=\lambda_i(\Ad_{g_i^{-1}}(t))=\kappa(t,y)=0
\]
since
\begin{enumerate}
\item the first equality is the definition,
\item  the second equality is the same logic with Equation \eqref{eq:lambdaandkappa}, 
\item the last equality comes from $t\in \Ad_{g_i}(\ft)\subset \Ad_{g_i}(\fb)$ and $y\in \Ad_{g_i}(\fn)$.
\end{enumerate}
 This gives that $\lambda_i=0$ since $\tilde{\phi}_i$ is surjective on $\ft$ for $i=1,2,\ldots, {\ell}$. Therefore, we have checked that $\mathrm{ker}(\phi^*)=\{0\}$, and so the map $\phi$ is surjective.    \end{proof}

\subsection{The converse direction of Theorem \ref{thm:main-abs}}\label{ss:converse-thm-1}
We prove the following sentence in this subsection. \begin{thm}\label{thm:between2and3}
	Let us consider regular elements $s_1, \ldots,s_{\ell} \in \ft $ such that the tuple  $(\mathrm{Ad}_G(s_1), \ldots,\mathrm{Ad}_G(s_{\ell}))$ is \emph{generic}.
For  \((g_1B, \ldots, g_{\ell}B) \in (G/B)^{\ell}\) with a given $g_1,\ldots , g_\ell\in G$, let us assume that there exists \((b_1,\ldots, b_{\ell}) \in B^{\ell}\) such that
		\begin{equation}\label{eq:equation-ad}
		\mathrm{Ad}_{g_1 b_1}(s_1) + \cdots + \mathrm{Ad}_{g_{\ell} b_{\ell}}(s_{\ell})
		= 0.
		\end{equation}
	Then $(g_1B,\ldots , g_\ell B)$ is absolutely indecomposable.
\end{thm}
 
  \subsubsection{Setting} Before starting the proof, we find an equivalent condition to prove Theorem \ref{thm:between2and3}.
  
\bigskip 	With the Bruhat decomposition of $G$, i.e., $G=UWB$, we have that \begin{equation}\label{eq:minor-change-eq}
  \begin{split}
  	&	\mathrm{Ad}_{g_1 b_1}(s_1) + \cdots+ \mathrm{Ad}_{g_{\ell} b_{\ell}}(s_{\ell})
  	= 0\\
  	\iff &	\mathrm{Ad}_{ b_1}(s_1) + \cdots + \mathrm{Ad}_{g_1^{-1}g_{\ell} b_{\ell}}(s_{\ell})=\mathrm{Ad}_{ b_1}(s_1) + \cdots+ \mathrm{Ad}_{u_{\ell}w_{\ell} b_{\ell}'}(s_{\ell})
  	= 0,
  \end{split}
  \end{equation}
  where $b_i,b_i'\in B, w_i \in W$ and $u_i \in U$ for each $i$.
  Now, let us decompose this as follows:
  \begin{equation}\label{eq:901}
  \begin{split}
  	& \mathrm{Ad}_{ b_1}(s_1) + \cdots+ \mathrm{Ad}_{u_{\ell} w_{\ell} b_{\ell}'}(s_{\ell})=0
  	\\ \iff& s_1 +n_1 + \cdots + \mathrm{Ad}_{  w_{\ell} }(s_{\ell})+\mathrm{Ad}_{ u_{\ell} w_{\ell} }(n_{\ell})+\tilde{n_3}=0
  	\\ \iff& s_1  +\cdots + \mathrm{Ad}_{  w_{\ell} }(s_{\ell})=-(n_1+\cdots+\mathrm{Ad}_{u_{\ell} w_{\ell} }(n_{\ell})+\tilde{n })\in \mathfrak{n}+\cdots +\mathrm{Ad}_{u_{\ell}w_{\ell}}(\mathfrak{n})
  \end{split}
  \end{equation}
  for some $n_1,\ldots,n_\ell, \tilde{n}\in \mathfrak{n}$. 
  Here, we use the following computation: For $s\in \ft$, $u\in U$, $w\in W$ and $b\in B$, we have
  \begin{equation}\label{eq:computationexample}
  \Ad_{uwb}(s)=\Ad_{uw}(s+n)=\Ad_{uw}(s)+\Ad_{uw}(n)=\Ad_{w}(s)+\Ad_{uw}(n)+\tilde{n},
  \end{equation}
  where $n \in \mathfrak{n}$ and $\tilde{n}=\Ad_{uw}(s)-\Ad_{w}(s)\in \mathfrak{n}$.
  Note that $\tilde{n}=\sum_{i=1}^{\ell} \left(\Ad_{u_iw_i}(s_i)-s_i\right)$. 
  
\bigskip
  From this decomposition and the definition of the generic condition, we can check the existence of $(b_1, \ldots , b_\ell)\in B^\ell$ and $(s_1, \ldots , s_\ell)\in \ft^\ell$ satisfying Equation \eqref{eq:equation-ad} is equivalent to show that
  \[
  \underset{\alpha\in \Psi}{\oplus} \check{\alpha}(k)\subset \mathfrak{n}+\cdots +\mathrm{Ad}_{u_{\ell}w_{\ell}}(\mathfrak{n}),
  \]
  where $\Psi$ is an isolated closed subsystem of $\Phi_G$.
 Here, $\check{\alpha}(k)$ is an one-parameter subalgebra $
  \check{\alpha}(k)
  =
  \bigl\{\,
  \check{\alpha}(s)
  \bigm|
  s \in k
  \bigr\}\subset  \ft $ for a coroot $\check{\alpha} $. 
  
 \bigskip

\subsubsection{Proof of Theorem \ref{thm:between2and3}} \label{s:fromatob}  
Now, let us prove Theorem \ref{thm:between2and3}. Let us assume that there is $t\in \mathrm{Stab}_G(B,u_{2}w_{2}B,\ldots ,u_{\ell}w_{\ell}B)^{ss}$ such that $C_G(t)$ is a proper Levi subgroup $L$. Without loss of the generality, we can assume that $C_G(t)$ is standard, i.e., $T\subset C_G(t)$. 
This is because as follows. There exist $v_1, \ldots, v_\ell\in B$ such that 
\[
\begin{split}
&t\in v_1Tv_1^{-1}\cap u_2w_2v_2Tv_2^{-1}w_2^{-1}u_2^{-1}\cap\ldots \cap u_{\ell}w_{\ell}v_{\ell}Tv_{\ell}^{-1}w_{\ell}^{-1}u_{\ell}^{-1}\\
\iff &  v_1Tv_1^{-1},\ u_2w_2v_2Tv_2^{-1}w_2^{-1}u_2^{-1},\ldots ,\ u_{\ell}w_{\ell}v_{\ell}Tv_{\ell}^{-1}w_{\ell}^{-1}u_{\ell}^{-1}\subset C_G(t) \\
\iff &T,\ v_1^{-1}u_2w_2v_2Tv_2^{-1}w_2^{-1}u_2^{-1}v_1,\ \ldots ,\ v_1^{-1}u_{\ell}w_{\ell}v_{\ell}Tv_{\ell}^{-1}w_{\ell}^{-1}u_{\ell}^{-1}v_1\subset C_G(v_1^{-1}tv_1) 
\end{split}
\]
from \cite[Proposition 3.5.2]{carter1985finite}. Then we can consider $\mathrm{Stab}_G(B,v_1^{-1}u_{2}w_{2}B,\ldots ,v_1^{-1}u_{\ell}w_{\ell}B)^{ss}$ instead of $\mathrm{Stab}_G(B,u_{2}w_{2}B,\ldots ,u_{\ell}w_{\ell}B)^{ss}$. Then by showing that $C_G(v_1^{-1}tv_1)$ is isolated, we can also check that $C_G(t)$ is also isolated.

\bigskip
So for convenience, let us assume that $v_1=1$.
Then we have that 
\[
 u_iw_iv_i=l_ix_i  \iff u_iw_i=  l_ix_i v_i^{-1}
\]
for some $l_i \in L$ and $x_i\in W$ for each $i$.
Let us apply this to Equation \eqref{eq:minor-change-eq}, then we have
\[
\begin{split}
&\mathrm{Ad}_{ b_1}(s_1) + \cdots+ \mathrm{Ad}_{u_{\ell}w_{\ell} b_{\ell}'}(s_{\ell}) 	
=  \mathrm{Ad}_{ b_1 }(s_1) + \mathrm{Ad}_{l_2x_2v_2^{-1} b_2'}(s_2)+\cdots + \mathrm{Ad}_{l_{\ell}x_{\ell}v_{\ell}^{-1} b_{\ell}'}(s_{\ell})=0.
\end{split}
\]
With the same logic in Equation \eqref{eq:901}, we can get that 
\begin{equation}\label{eq:911}
\begin{split}
	&s_1+m_1 + \mathrm{Ad}_{ l_2x_2 }(s_2)+\mathrm{Ad}_{ l_2x_2 }(m_2) + \cdots + \mathrm{Ad}_{ l_{\ell}x_{\ell} }(s_{\ell})+\mathrm{Ad}_{ l_{\ell}x_{\ell} }(m_{\ell})=0\\
	\Rightarrow &
	s_1  + \mathrm{Ad}_{ l_2x_2 }(s_2)+ \cdots + \mathrm{Ad}_{ l_{\ell}x_{\ell} }(s_{\ell})=
	-(m_1+\mathrm{Ad}_{ l_2x_2 }(m_2) +\cdots +\mathrm{Ad}_{ l_{\ell}x_{\ell} }(m_{\ell}) )
\end{split}
\end{equation}
where $\mathfrak{l}$ is the Lie algebra of $L$ and for some $m_1,\ldots,m_{\ell}\in\mathfrak{n}$.
Then since $
	m_1+\mathrm{Ad}_{ l_2x_2 }(m_2) +\cdots +\mathrm{Ad}_{ l_{\ell}x_{\ell} }(m_{\ell}) \in  \mathfrak{n}+\mathrm{Ad}_{ l_2x_2 }(\mathfrak{n}) +\cdots+\mathrm{Ad}_{ l_{\ell}x_{\ell} }(\mathfrak{n})  $, we have that 
	\[
		s_1  + \mathrm{Ad}_{ l_2x_2 }(s_2)+ \cdots + \mathrm{Ad}_{ l_{\ell}x_{\ell} }(s_{\ell})\in   \mathfrak{n}+\mathrm{Ad}_{ l_2x_2 }(\mathfrak{n}) +\cdots+\mathrm{Ad}_{ l_{\ell}x_{\ell} }(\mathfrak{n}).
	\]

 Now, let us apply the Bruhat decomposition to $L$ again. Since $L$ is standard, we have the following decomposition $l_i=h_i y_i h_i'$  for some $h_i,\in U\cap L$, $h_i'\in B\cap L$ and $y_i\in W_L$ for each $i$.
 
 Then repeating Equation \eqref{eq:computationexample}, we have the following:
 \begin{equation}\label{eq:sumofnilpotentsub}
 s_1+\Ad_{y_2x_2}(s_2)+\cdots +\Ad_{y_{\ell}x_{\ell}}(s_{\ell})\in \mathfrak{n} 
 +\sum_{i=2}^{\ell} \mathrm{Ad}_{ l_ix_i }(\mathfrak{n})  
 +\sum_{i=2}^{\ell} \mathrm{Ad}_{h_iy_i}(\mathfrak{n}_L),
 \end{equation}
 where $\mathfrak{n}_L=Lie(U\cap L)$.

\bigskip 
Our main focus is to consider the coroot space, i.e., $\check{\alpha}(k)$ in $  \mathfrak{n} 
+\sum_{i=2}^{\ell} \mathrm{Ad}_{ l_ix_i }(\mathfrak{n}) +\sum_{i=2}^{\ell} \mathrm{Ad}_{h_iy_i}(\mathfrak{n}_L)
$.
Since $l_i \in L$, $x_i,y_i \in W$ and $h_i \in U\cap L$ for each $i$,
 the space
$ \mathfrak{n} 
+\sum_{i=2}^{\ell} \mathrm{Ad}_{ l_ix_i }(\mathfrak{n}) +\sum_{i=2}^{\ell} \mathrm{Ad}_{h_iy_i}(\mathfrak{n}_L)
$  contains (at most) coroot spaces coming from $L$. In other words,  we can get that
\[
\mathfrak{t}\cap \left(\mathfrak{n} 
+\sum_{i=2}^{\ell} \mathrm{Ad}_{ l_ix_i }(\mathfrak{n})  
 +\sum_{i=2}^{\ell} \mathrm{Ad}_{h_iy_i}(\mathfrak{n}_L)\right)
\subset \underset{\alpha\in \Phi_L}{\oplus}\check{\alpha}(k)
\]
since $\Phi_L$ is a closed subsystem of $\Phi$.
This implies that we can check that
\[
s_1  + \mathrm{Ad}_{ y_2x_2 }(s_2)+ \cdots  \mathrm{Ad}_{ y_{\ell}x_{\ell} }(s_{\ell}) \in  \mathfrak{t} \cap\left(\mathfrak{n} 
+\sum_{i=2}^{\ell} \mathrm{Ad}_{ l_ix_i }(\mathfrak{n}) +\sum_{i=2}^{\ell} \mathrm{Ad}_{h_iy_i}(\mathfrak{n}_L)\right)
\subset \mathfrak{t}\cap [\mathfrak{l},\mathfrak{l}]
\]
with  Equation \eqref{eq:sumofnilpotentsub}.
However, the generic condition gives that
\[
s_1  + \mathrm{Ad}_{ y_2x_2 }(s_2)+ \cdots + \mathrm{Ad}_{ y_{\ell}x_{\ell} }(s_{\ell}) \notin    \mathfrak{t}\cap [\mathfrak{l},\mathfrak{l}],
\]
for a proper Levi subgroup $L$, and so this is a contradiction. Thus every $t\in \mathrm{Stab}_G\bigl(g_1 B, \ldots, g_{\ell} B\bigr)^{ss}$ satisfies that $C_G(t)$ is isolated, and so $(g_1B, \ldots , g_\ell B)$ is absolutely indecomposable. \qed

    \subsection{Generalisation}\label{ss:generisation} 
    Let $\mathcal{P}:=(P_1, \ldots , P_\ell)$ be a tuple of parabolic subgroups of $G$. We say that $\mathcal{P}$ is \emph{generic} when there exists a generic tuple $(s_1, \ldots , s_\ell)\in \ft^\ell$ such that $C_G(s_i)=L_i$, where $P_i=L_i  U_i$. 
From this definition, we call a tuple $(s_1, \ldots , s_\ell)\in \ft^\ell$ is generic $\mathcal{P}$-type if the tuple is generic and $C_G(s_i)=L_i$ for all $i$.   When $G=\GL_n$ and the characteristic is sufficiently large, a tuple $\mathcal{P}$ is generic if and only if the corresponding partitions are indivisible, cf. \cite[\S1.3]{hausel2011arithmetic}. In this paper, we only consider generic $\mathcal{P}$. Then we have the following extension:

   \begin{thm}\label{thm:extension}
   An element $(g_1P_1, \ldots , g_\ell P_\ell)\in G/P_1\times \ldots \times G/P_\ell$ is absolutely indecomposable for a given $g_1, \ldots , g_\ell\in G$ if and only if there exists $(p_1, \ldots , p_\ell) \in \mathcal{P}$ such that 
   \[
   \Ad_{g_1p_1}(s_1)+\cdots+   \Ad_{g_\ell p_\ell}(s_\ell)=0
   \]
   for a tuple of generic semisimple $\mathcal{P}$-type adjoint orbits $(\Ad_G(s_1),\ldots , \Ad_G(s_\ell))$.
   \end{thm}
   \begin{proof}
   Let us assume that $(g_1P_1, \ldots , g_\ell P_\ell)$ is absolutely indecomposable. Let $\ft_i:=\{t\in \ft\,|\, C_G(t)\supset L_i\}$ and $\fn_i$ be the nilpotent radical of $\mathfrak{p}_i$.  When $P_i=B$ for some $i$, then $\ft_i=\ft $. Note that $\ft_i$ is a subalgebra of $\ft$ since every element $t$ in $\ft_i$ satisfies that $\alpha( t)=0$ for all $\alpha \in \Phi_{L_i}$. 
   Then following Theorem \ref{conj:between2and3}, let us consider
 \[
 \phi_\mathcal{P}\,:\, X_\mathcal{P}\rightarrow \times_{i=1}^\ell \ft_i \quad \text{given by }\ \phi((X_1, \ldots , X_\ell))=(\Ad_{g_1^{-1}}(\tilde{t_1}),\ldots , Ad_{g_\ell^{-1}}(\tilde{t_\ell})),
 \]
 where 
 \[
 X_\mathcal{P}:=\{ (X_1, \ldots, X_\ell)\in \Ad_{g_1}(\ft_1+\fn_1)\times \ldots \times \Ad_{g_\ell}(\ft_\ell+\fn_\ell)\,|\, X_1+\cdots +X_\ell=0\}.
 \]
Let us take a generic tuple $(s_1,  \ldots , s_\ell)\in \times_{i=1}^\ell \ft_i $ such that $C_G(s_i)=L_i$. Then since every component in $\phi_\mathcal{P}^{-1}((s_1,  \ldots , s_\ell))$ is again semisimple, we can prove this direction by showing that $\phi_\mathcal{P}$ is surjective into $ \times_{i=1}^\ell \ft_i$ whose sum is in $[\fg,\fg]$. The surjectivity can be shown by following the proof of Lemma \ref{lem:dominant} with $ \phi_\mathcal{P}$ and $X_\mathcal{P}$ instead of $\phi$ and $X$.
The proof of the converse direction is the exactly same with the proof of Theorem \ref{thm:between2and3} since $\Ad_{p_i}(s_i)=s_i+n_i$ for some $n_i \in \fn_i$ due to the assumption that $C_G(s_i)=L_i$.
   \end{proof}

\subsection{The number of absolutely indecomposable quasi-parabolic $G$-bundles}\label{s:analogueof5.7}
Now, let us consider the number of  isomorphism classes of absolutely indecomposable
quasi-parabolic $G$-bundles  over a
finite field.

Let $(\theta_1,\ldots , \theta_\ell)$ be a   tuple of linear characters of $(L_1^F,\ldots , L_\ell^F)$ for parabolic subgroups $P_i=L_iU_i$. For $p=(g_1P_1,\ldots , g_\ell P_\ell)\in (G/P_1,\ldots , G/P_\ell)$ with $g_i \in G^F$, we can define 
\[
\Theta_{p,(\theta_1,\ldots , \theta_\ell)}(h):=\prod_{i=1}^\ell \widetilde{\theta_i}(g_i^{-1}hg_i)\quad \text{for }h\in \mathrm{Stab}_G(p)^F,
\]
similar to Equation \eqref{eq:central-character}.
\begin{lem} \label{lem:averaging}
We have
\begin{equation}\label{eq:averaging}
 \left\langle\bigotimes_iR_{L_i}^{G}(\theta_i),1\right\rangle_{G^F}
 =\sum_{[p]\in (G^F/P_1^F,\ldots , G^F/P_\ell^F)/G^F}
     \frac{1}{|\mathrm{Stab}_G(p)^F|}\sum_{h\in \mathrm{Stab}_G(p)^F}\Theta_{p,(\theta_1,\ldots , \theta_\ell)}(h).
\end{equation}
\end{lem}
Note that each summand $  \frac{1}{|\mathrm{Stab}_G(p)^F|}\sum_{h\in \mathrm{Stab}_G(p)^F}\Theta_{p,(\theta_1,\ldots , \theta_\ell)}(h)$ is $1$ if $\Theta_{p,(\theta_1,\ldots , \theta_\ell)}=1$, and $0$ otherwise.

\begin{proof}
Recall that $R_{L_i}^G(\theta_i)(h)=\mathrm{Ind}_{P_i}^G(\widetilde{\theta_i})(h)$ since every $P_i$ is $F$-stable. Then using the induced character formula, we have $R_{L_i}^G(\theta_i)(h)=\frac{1}{|P_i^F|}\sum_{\substack{g_i\in G^F\\ g_i^{-1}hg_i\in P_i^F}} \widetilde{\theta_i}(g_i^{-1}hg_i)$. Since $\widetilde{\theta_i}$ is invariant under $P_i^F$-conjugation, we can get that 
\[
 R_{L_i}^{G}(\theta_i)(h)
 =\sum_{\substack{g_iP_i\in G^F/P_i^F\\h g_iP_i=g_iP_i}}
       \widetilde\theta_i(g_i^{-1}hg_i).
\]
By multiplying over $i$ and averaging over $G^F$, we get
\[
 \frac{1}{|G^F|}\sum_{p\in (G^F/P_1^F,\ldots , G^F/P_\ell^F)}
       \sum_{h\in \mathrm{Stab}_G(p)^F}\Theta_{p,(\theta_1,\ldots , \theta_\ell)}(h).
\]
Then by grouping points over $G^F$-orbits, orbit-stabiliser theorem gives the desired result.
\end{proof}
Note that Lemma \ref{lem:averaging} is true although the tuple $(\theta_1, \ldots , \theta_\ell)$ is not generic. However, if this is generic, then we get the following result.
\begin{thm} \label{thm:main}
Let us assume that there exists a tuple $(\theta_1, \ldots , \theta_\ell)$ is generic over $(L_1^F,\ldots , L_\ell^F)$ as Definition \ref{def:generic-char}. Then we have
\begin{equation}\label{eq:main}
 A_{G^F,P_1\times \cdots \times P_\ell}(q)
 =\left\langle\bigotimes_{i=1}^{\ell}R_{L_i}^{G}(\theta_i),1
       \right\rangle_{G^F}.
\end{equation}
\end{thm}
Note that each component $R_{L_i}^{G}(\theta_i)$ does not need to be irreducible.

\begin{proof}
By Lemma~\ref{lem:averaging}, it suffices to show
\begin{equation}\label{eq:selector}
\Theta_{p,(\theta_1,\ldots , \theta_\ell)}=1\quad\Longleftrightarrow\quad p\text{ is absolutely indecomposable}.
\end{equation}

Suppose   that $p\in (G^F/P_1^F,\ldots,G^F/P_\ell^F) $ is absolutely indecomposable, and take $h\in \mathrm{Stab}_G(p)^F$.
Its Jordan decomposition $h=h_sh_u=h_uh_s$ has $h_s,h_u\in H_x^F$ from \cite[Theorem 9.18]{Milne}. Then from condition (i) of Definition \ref{def:generic-char}, we can get that $\Theta_{p,(\theta_1,\ldots , \theta_\ell)}(h)=\Theta_{p,(\theta_1,\ldots , \theta_\ell)}(h_s)$, because the
homomorphisms $\pi_i:P_i\rightarrow L_i$ send unipotent elements to unipotent elements.
Let us take a maximal torus $T_i$ contained in $g_iP_ig_i^{-1}$ which contains $h_s$. Then since $Z_{C_G(h_s)}\subset T_i \subset g_iP_ig_i^{-1}$ for all $i$, we can check that $Z_{C_G(h_s)}\subset \mathrm{Stab}_G(p)$. Then from the definition of absolutely indecomposable quasi-parabolic $G$-bundle, $C_G(h_s)$ is isolated, and then we get that 
\[
\Theta_{p,(\theta_1,\ldots , \theta_\ell)}(h)=\Theta_{p,(\theta_1,\ldots , \theta_\ell)}(h_s)=1
\]
from condition (iii).
Hence $\Theta_{p,(\theta_1,\ldots , \theta_\ell)}$ is trivial on $\mathrm{Stab}_G(p)^F$.

Conversely, suppose $p\in (G^F/P_1^F,\ldots , G^F/P_\ell^F)$ is not absolutely indecomposable. The   connected  group
$\mathrm{Stab}_G(p)^\circ$ has a $F$-stable maximal torus $S$, cf.
\cite[Proposition~17.99]{Milne}. Since $Z_G^\circ$ is central,
$Z_G^\circ\subseteq S$. The assumption that $p$ is not absolutely indecomposable implies
\[
 \dim S>\dim Z_G^\circ.
\]
Note that $C_G(S)$ is a proper Levi subgroup of $G$. Now, with the same logic of above, let us take a maximal torus $T_i$ of $g_iP_ig_i^{-1}$ such that $Z_{C_G(S)}\subset T_i \subset g_iP_ig_i^{-1}$, and then this implies that $Z_{C_G(S)}\subset \mathrm{Stab}_G(p)$. Then from condition (ii), this gives that
\[
\Theta_{p,(\theta_1,\ldots ,\theta_\ell)}|_{Z_{C_G(S)}^F}\neq 1,
\]
and this finishes the proof.
\end{proof}

  \section{Multiplicity of generic regular semisimple characters} \label{s:regularsemisimplechar}
 
In this section, we  consider Theorem \ref{thm:maininintro}.
Let us assume that $\theta_1, \ldots , \theta_{\ell}\in 
\widehat{T^F}$ are \emph{split}, \emph{generic}, and \emph{in general position}. Here, split means that $T$ is a split maximal torus of $G$. Recall that we can check that the existence of such tuple under the in general position assumption.
In this case, we have that the Deligne-Lusztig characters $R_T^G(\theta_i) = \mathrm{Ind}_{B^F}^{G^F}(\widetilde{\theta_i})$ are irreducible, so let us denote them as $\chi_{\theta_i}$. Note that these characters are called \emph{semisimple character} in general, cf. \cite[\S2.6]{geck2020character}.

\subsection{Character values}
We use the following well-known result about the value of the Deligne-Lusztig character.
\begin{thm}\cite[Theorem 2.2.16]{geck2020character}\label{thm:charactervalue}
    Let $g\in G^F$ and $g=g_sg_u=g_ug_s$, where $g_s$ is semisimple and $g_u$ is unipotent. Let $H=C_G(g_s)$. Then 
    \[
    R_T^G(\theta)(g)=\frac{1}{|H^F|}\sum_{\substack{x\in G^F\\ x^{-1}g_sx\in T}}Q_{xTx^{-1}}^H(g_u)\theta(x^{-1}g_sx).
    \]
\end{thm}
 Recall that $H=C_G(g_s)$ is connected for any semisimple element $g_s$ since the derived subgroup of $G$ is simply connected, see \cite[\S2.2.6]{KNWG}.

\begin{rem}\label{rem:vanishingdeligne}From \cite[Example 2.2.17 (a)]{geck2020character},
when $\theta\in \widehat{T^F}$, if $g_s \notin T^F $ (up to $G^F$-conjugation) for $g=g_sg_u=g_ug_s$, then $R_T^G(\theta)(g_g)=0$. Therefore,   we need to consider $g_s\in T^F$ (up to $G^F$-conjugation) for $g=g_sg_u=g_ug_s$.
\end{rem}

For ease of computation, we need the following lemma.
\begin{lem}\label{lem:decompositionofcen}

    For a semisimple element $t$ in $T^F$, we have  \begin{equation}\label{eq:weyl-iso}\{ x\in G^F \,|\, xtx^{-1}\in T^F\}=\underset{v\in N_G(T)^F /T^F }{\cup}\dot{v} H^F=\underset{v\in  N_G(T)^F /N_{H}(T)^F }{\sqcup}\dot{v} H^F=\underset{v\in  W  /W_{H}  }{\sqcup}\dot{v} H^F ,
    \end{equation} 
     where $\dot{v}$ is a representative of $v  $, $H=C_G(t)$ and $W_H$ is the Weyl group of $T$ in $H$.
\end{lem}

\begin{proof}  It is obvious  $\underset{v\in N_G(T)^F /T^F }{\cup}\dot{v}{H}=\underset{v\in N_G(T)^F /N_{{H}}(T)^F }{\sqcup}\dot{v}{H }$, so  this gives the second equality in Equation \eqref{eq:weyl-iso}. Let us prove the last equality in Equation \eqref{eq:weyl-iso}. Since $q$ is sufficiently large so that every maximal torus $T^F$ is non-degenerate, we obtain $$N_G(T)^F /N_{{H}}(T)^F\simeq (N_G(T)^F/T^F) /(N_{{H}}(T)^F/T^F)\simeq W^F/W_{H}^F$$ from \cite[Corollary 3.6.5]{carter1985finite}. Since $T$ is split, we have $W^F=W$ and $W_{H}^F=W_H$, 
 and this gives the last equality.

Now, let us show the first equality in Equation \eqref{eq:weyl-iso}. The inclusion $\{ x\in G^F \,|\, xtx^{-1}\in T^F\}\supset \underset{v\in N_G(T)^F /N_{H}(T)^F }{\sqcup}\dot{v}{H^F}$ is easy to check, so let us consider the converse inclusion by showing that any element $x\in G^F$ satisfying $xsx^{-1}\in T^F$ can be written as $\dot{v}h$ for some $v\in N_G(T)^F /N_{H}(T)^F $ and $h \in H^F$. 
Recall that when $xtx^{-1}\in T$, we have $x^{-1}Tx\subset H$, cf. \cite[Proposition 3.5.2]{carter1985finite}. This implies that $T, x^{-1}Tx\subset H$, and so we can find $h\in H$ such that $hx^{-1}Txh^{-1}=T$. Therefore, \[
    hx^{-1} \in N_G(T)\Rightarrow xh^{-1}\in N_G(T)\Rightarrow x\in N_G(T) H\simeq (N_G(T)/N_{H}(T)) H.\]
Then we have a decompose $x=\dot{v}h $ for some $v\in N_G(T)/N_{H}(T)$ and $h\in H$. 
We can finish the proof from the fact  $\{ \dot{v}\in G \,|\, v\in N_G(T)/N_{H}(T)\} \cap H=\{1\}$.   Since $T$ is split, we can choose $\dot{v}$ as an element of $N_G(T)^F$. Then since $x\in G^F$, we have that $h\in H^F$. Therefore, we are done.
\end{proof}

\begin{cor}
\label{cor:charactervalue}
    With the same notations in Theorem \ref{thm:charactervalue},  
    \[
    R_T^G(\theta)(g)=\frac{Q_T^H(g_u)}{|W_H|}\sum_{w\in W}(w\cdot \theta)(g_s).
    \]
    Here, $W_H$ is the Weyl group of $H=C_G(g_s)$.  
\end{cor}
\begin{proof}
    This can be proved using Lemma \ref{lem:decompositionofcen} with the fact that $Q_{T}^G=Q_{xTx^{-1}}^G$ for any $x \in G^F$, cf. \cite[Definition 2.2.15]{geck2020character}.
\end{proof}

\subsection{Multiplicity} Let us compute the term $\langle \Lambda\otimes  \chi_{\theta_1}\otimes \ldots \otimes \chi_{\theta_{\ell}},1\rangle$ appearing in Theorem \ref{thm:maininintro}. To compute this, we define types of elements in $G$. This is a multiplicative analogue of $\fg$-types in \cite[\S6.2.3]{KNWG}.
 The type of $g\in G^F$ is the tuple $([\Phi_{C_G(g_s)}],[g_u])$, where $\Phi_{C_G(g_s)}$ is the root subsystem of $C_G(g_s)$ in $\Phi$, and $[\ ]$ denotes its $W$-orbit and $C_G(s)^F$-orbit respectively. Let us denote the set of types of $G$ by $\mathcal{T}(G)$. Note that the set $\mathcal{T}(G)$ is finite, and independent on $k$ when the characteristic is very good. 
Then we can define the type map $\xi \,:\, G\rightarrow \mathcal{T}(G)$ given by $\xi(g)=([\Phi_{C_G(g_s)}],[g_u])$.

Then we have
\[\begin{split}
&|G^F|\langle \Lambda\otimes \chi_{\theta_1}\otimes \ldots \otimes \chi_{\theta_{\ell}},1\rangle\\
=&\sum_{g\in G} \Lambda(g)\chi_{\theta_1}(g)\cdots \chi_{\theta_{\ell}}(g)\\
=&\sum_{\tau \in \mathcal{T}(G)}\sum_{g\in \xi^{-1}(\tau)}  \Lambda(g) R_T^G(\theta_1)(g)\cdots R_T^G(\theta_{\ell})(g)
\\=&\sum_{\tau =([\Phi_L],[u])\in \mathcal{T}(G)} \frac{|W_L|}{|W|} q^{g\dim(H_\tau)}\sum_{\substack{s\in T^F \\ [\Phi_{C_G(s)}]=[\Phi_E]}} |\mathcal{O}_\tau| \frac{Q_T^L(u)^{\ell}}{|W_L|^{\ell}} \sum_{(w_1, \ldots ,w_{\ell})\in W^{\ell}} \prod_{i=1}^{\ell}(w_i\cdot\theta_i)(s)
\\=&\sum_{\tau =([\Phi_L],[u])\in \mathcal{T}(G)} \frac{ q^{g\dim(H_\tau)} |\mathcal{O}_\tau|Q_T^L(u)^{\ell}}{|W_L|^{\ell-1}|W|} \sum_{(w_1, \ldots ,w_{\ell})\in W^{\ell}}  \sum_{\substack{s\in T^F \\ [\Phi_{C_G(s)}]=[\Phi_L]}}\prod_{i=1}^{\ell}(w_i\cdot \theta_i)(s),
\end{split}\]
where $\mathcal{O}_\tau$ is the conjugacy class and $H_\tau=C_G(g)$ for some $g \in \xi^{-1}(\tau)$.
From Remark \ref{rem:vanishingdeligne}, we only need to consider $g\in G^F$ such that $g=g_sg_u=g_ug_s$ with $g_s\in T^F$ (up to conjugation), and so we only need to consider the root systems of pseudo-Levi subgroups containing $T$.

Let us change the last equation with the same reason in \cite[\S4.2.2]{KNWG} as follows:
\[\begin{split}
&\sum_{\tau =([\Phi_L],[u])\in \mathcal{T}(G)} \frac{ q^{g\dim(H_\tau)} |\mathcal{O}_\tau|Q_T^L(u)^{\ell}}{|W_L|^{n-1}|W|} \sum_{(w_1, \ldots ,w_{\ell})\in W^{\ell}}  \sum_{\substack{s\in T^F \\ [\Phi_{C_G(s)}]=[\Phi_L]}}\prod_{i=1}^{\ell}(w_i\cdot\theta_i)(s)\\
&=\sum_{\tau  \in \mathcal{T}(G)} \frac{ q^{g\dim(H_\tau)} |\mathcal{O}_\tau|Q_T^L(u)^{\ell}}{|W_L|^{n-1}|N_W(W_L)| } \sum_{(w_1, \ldots ,w_{\ell})\in W^{\ell}}  \sum_{\substack{s\in T^F \\ \Phi_{C_G(s)}=\Phi_L}}\prod_{i=1}^{\ell}(w_i\cdot\theta_i)(s).
\end{split}
\]

Then our next goal is to compute the last sum $$ \sum_{\substack{s\in T^F\\ \Phi_{C_G(s)}=\Phi_L}}\prod_{i=1}^{\ell}(w_i\cdot\theta_i)(s).$$  Let us denote the set of standard pseudo-Levi subsystems of $G$ as $\mathcal{E}(G)$, and this set is again finite obviously. Let us consider this set as a poset using the inclusion. Then we can use the M\"obius inversion on the poset with the M\"obius function $  \mu_{\mathcal{E}(G)}$ (cf. \cite{KNP,KNWG}), and so we have the following:
\[
\sum_{\substack{s\in T^F\\ \Phi_{C_G(s)}=\Phi_L}}\prod_{i=1}^{\ell}(w_i\cdot \theta_i)(s)=\sum_{\substack{\Phi_H \in \mathcal{E}(G)\\ \Phi_H \supset \Phi_L}}  \mu_{\mathcal{E}(G)}(\Phi_E,\Phi_H)\sum_{\substack{s\in T^F\\ \Phi_{C_G(s)} \supset \Phi_H}}\prod_{i=1}^{\ell}(w_i\cdot \theta_i)(s)
\]
since
\[
\sum_{\substack{s\in T^F\\\Phi(C_G(s) )\supset \Phi_L}}\prod_{i=1}^{\ell}(w_i\cdot\theta_i)(s) = \sum_{\substack{\Phi_H\in \mathcal{E}(G)\\ \Phi_H \supset \Phi_L}}\sum_{\substack{s\in T^F\\ \Phi_{C_G(s)}=\Phi_H}} \prod_{i=1}^{\ell}(w_i\cdot\theta_i)(s).
\]
Now, we use the same technique which is utilised in \cite[\S4]{KNWG}.
With the fact that $
\{s\in T^F\,|\, C_G(s)\supset H\}=Z_H^F,$  where $Z_H $ is the centre of $H$, we have that 
\[
\sum_{\substack{s\in T^F\\ \Phi_{C_G(s)} \supset \Phi_H}}\prod_{i=1}^{\ell}(w_i\cdot\theta_i)(s)=  \sum_{\substack{s\in Z_H^F}}\prod_{i=1}^{\ell}(w_i\cdot\theta_i)(s) = \begin{cases}
    |Z_H^F|\quad &\text{if } H\ \text{is isolated}\\
    0& \text{otherwise}
\end{cases}
\]
from the definition of genericity of $(\theta_1, \ldots , \theta_{\ell})$.
Therefore, we have the following result:
\[
\sum_{\substack{s\in T^F\\ \Phi_{C_G(s)}=\Phi_L}}\prod_{i=1}^{\ell}(w_i\cdot\theta_i)(s)=\sum_{\substack{\Phi_E\in \mathcal{E}(G)\\ E\text{ is isolated}}}  \mu_{\mathcal{E}(G)}(\Phi_L,\Phi_E)|Z_E^F|.
\]
Note that this value is the same for any $(w_1,\ldots ,w_{\ell})\in W^{\ell}$. Furthermore, $\mu_{\mathcal{E}(G)}(\Phi_L,\Phi_E)=0$ if $E$ does not contain $L$. 
Therefore, we have the following formula:
\begin{equation}\label{eq:resultmultiplicity}
   \begin{split}
  &   \langle \Lambda \otimes\chi_{\theta_1}\otimes\ldots \otimes \chi_{\theta_{\ell}},1\rangle\\
  =& \frac{1}{|G^F|}\sum_{\tau \in \mathcal{T}(G)}  \frac{q^{g\dim(H_\tau)}|\mathcal{O}_\tau|Q_T^L(u)^{\ell} }{|N_W(W_L)||W_L|^{\ell-1}}  \sum_{(w_1, \ldots ,w_{\ell})\in W^{\ell}}\sum_{\substack{E\in \mathcal{E}(G)\\ E\text{ is isolated}}}  \mu_{\mathcal{E}(G)}(\Phi_L,\Phi_E)|Z_E^F|\\
     =&\frac{1}{|G^F|}\sum_{\tau \in \mathcal{T}(G)}  \frac{q^{g\dim(H_\tau)} |\mathcal{O}_\tau|Q_T^L(u)^{\ell}  |W|^{\ell}}{|N_W(W_L)| |W_L|^{\ell-1}} \sum_{\substack{E\in \mathcal{E}(G)\\ E\text{ is isolated}}}  \mu_{\mathcal{E}(G)}(\Phi_L,\Phi_E)|Z_E^F|\\
     =& \sum_{\tau \in \mathcal{T}(G)}  \frac{q^{g\dim(H_\tau)} Q_T^L(u)^{\ell}   |W|^{\ell-1}|[\Phi_L]|}{|C_L(u)^F||W_L|^{\ell-1}} \sum_{\substack{E\in \mathcal{E}(G)\\ E\text{ is isolated}}}  \mu_{\mathcal{E}(G)}(\Phi_L,\Phi_E)|Z_E^F|
\end{split}
\end{equation}
 from the fact that  $|\mathcal{O}_\tau|=\frac{|G^F|}{|C_L(u)^F|}$ and $\frac{1}{|N_W(W_L)|}=\frac{|[\Phi_L]|}{|W|}$.

 \subsection{The additive character variety}
 Let us recall the size of additive character varieties from \cite[\S6]{KNWG}. Similar to types of $G$, let us define types of $\fg$. In \cite{KNWG}, the authors used subgroups, but here we replace subgroups by subsystems.
\begin{defe}
The types of $\fg$ is the set of tuples $([\Phi_L],[n])$, denoted by $\mathcal{T}(\fg)$, where $[\Phi_L]$ is $W$-orbit of the root subsystem of a standard Levi subgroup $L$ in $G$, and $[n]$ is an adjoint orbit of a nilpotent element $n$ in $\mathfrak{l}^F=\mathrm{Lie}(L)^F$.
\end{defe}

With this definition, we can get the following result. Let $\varpi$ be the Springer isomorphism, i.e., a $G$-equivariant isomorphism from the nilpotent cone of $\fg$ to the unipotent variety of $G$. This isomorphism exists when $\mathrm{char}(k)$ is very good for $G$ from \cite[\S 2.7.5]{letellier2005fourier}.  
\begin{thm} Let us assume that $ (O_1, \ldots , O_{\ell})$ is a tuple of generic split regular semisimple orbits. Then we have the following result from \cite[\S6.6]{KNWG}.
	
	\begin{equation}\label{eq:additive}
       |\mathcal{X}_\fg^F|=	|(Z_G^\circ)^F|q^{\gamma_G} \sum_{\tau=([\Phi_L],[n]) \in \mathcal{T}(\fg)}  \frac{q^{g\dim(H_\tau)}|[\Phi_L]|  |W|^{\ell-1} Q_T^L(\varpi(n))^{\ell}}{|W_L|^{\ell-1}|C_L(n)^F|} 
	 \mu_{\mathcal{L}(G)}(\Phi_L,\Phi)  ,
	\end{equation}
	where $\gamma_G= \dim(Z_G) +(g-1)\dim(\fg) + \ell|\Phi^+|$, and $\mathcal{L}(G)$ is the set of standard Levi subsystems of $G$ and $\mu_{\mathcal{L}(G)}$ the M\"obius function on the poset $\mathcal{L}(G)$.\end{thm}
Recall that  $\dim(H_\tau)=\dim(C_L(\varpi(n)))=\dim(C_L(n))$ since $\varpi$ is $G$-equivariant.

\begin{rem}
Note that in \cite{KNWG}, authors consider the case that the centre of $G$ is connected. However, the almost computation works also to our case, i.e.,  the derived subgroup of $G$ is simply connected. The difference is that we could not guarantee the polynomial count as in \cite[\S6.6.2]{KNWG}, and we used the fact that $\left|(G/Z_G)^F\right|=\frac{|G^F|}{|(Z_G^\circ)^F|}$.
 \end{rem}

 \subsection{Properties of M\"{o}bius function}
 To compare $\langle \Lambda \otimes \chi_{\theta_1}\otimes \ldots \otimes \chi_{\theta_{\ell}},1\rangle $ (Equation \eqref{eq:resultmultiplicity}) and 
$       |\mathcal{X}_\fg^F|$ (Equation \eqref{eq:additive}), we need information about M\"obius functions on $\mathcal{E}(G)$ and $\mathcal{L}(G)$. Let us introduce such information in this subsection. 
Recall that we have two M\"{o}bius functions in this paper.
\begin{enumerate}
    \item Let $\mathcal{E}(G)$ be the poset (via inclusion) of standard pseudo-Levi subsystems in $G$, then we have the corresponding M\"obius function, denoted by $\mu_{\mathcal{E}(G)}$.
    \item Let $\mathcal{L}(G)$ be the poset (via inclusion) of standard Levi subsystems in $G$, then we have the corresponding M\"obius function, denoted by $\mu_{\mathcal{L}(G)}$.
\end{enumerate}

\subsubsection{Galois connection theorem}
Note that for any pseudo-Levi subsystem $\Psi$ of $\Phi$, there exists a unique minimal Levi subsystem $\overline{\Psi}$ of $\Phi$ containing $\Psi$. This is because if $\Psi=w\cdot \langle S \rangle $ for some $S\subset \tilde{\Delta}$, we can take a Levi subsystem $\mathrm{Span}_{\mathbb{R}}(w\cdot S)\cap \Phi$. This is minimal due to the dimension and unique from the construction.

Let us take maps $f\,:\, \mathcal{L}(G)\rightarrow \mathcal{E}(G)$ is the inclusion map and  $g\,:\, \mathcal{E}(G)\rightarrow \mathcal{L}(G)$ given by $g(\Psi)=\overline{\Psi}$ when $\overline{\Psi}$ is the minimal unique Levi subsystem containing $\Psi$.   Note that $  \Omega\leq f(\Psi) \iff  g(\Omega)\leq \Psi $ for $\Psi\in \mathcal{L}(G)$ and $\Omega \in \mathcal{E}(G)$.
Note that taking $g(\Psi)=\overline{\Psi}$ satisfies the closure operator conditions in \cite[Page 45]{kung2009combinatorics}, and the set of closed points of $\mathcal{E}(G)$ is $\mathcal{L}(G)$. 
 Then we have the following result.

\begin{lem}\label{lem:galoisconnection}
For a Levi subsystem $\Psi$ of $\Phi$, we have
\[
\sum_{\substack{\Omega\in \mathcal{E}(G)\\f(\Omega)=\Phi }}\mu_{\mathcal{E}(G)}(\Psi,\Omega)= \mu_{\mathcal{L}(G)}(\Psi,\Phi).
\]
If $\Psi$ is pseudo-Levi, but not Levi, then we have $\sum_{\substack{\Omega\in \mathcal{E}(G)\\f(\Omega)=\Phi }}\mu_{\mathcal{E}(G)}(\Psi,\Omega)=0$.
\end{lem}
\begin{proof}
This is the Galois connection theorem in \cite[\S 3.1.8]{kung2009combinatorics} with the fact that the set of closed points of $\mathcal{E}(G)$ is $\mathcal{L}(G)$.
\end{proof}

\subsubsection{Restriction of the M\"obius function}
Let us consider a proper isolated pseudo-Levi subsystem $\Phi_E$ and
\[
\mathcal{E}(E):=\{ \Psi\subset \Phi_E\,|\, \Psi \in \mathcal{E}(G)\}.
\]
Then $\mathcal{E}(E)$ is a subset of pseudo-Levi subsystems of $E$. This is  a convex subposet of $\mathcal{E}(G)$, which means that whenever $\Psi_1,\Psi_2\in \mathcal{E}(E)$ and $\Psi_1 \subset \Psi \subset \Psi_2$ for some $\Psi \in \mathcal{E}(G)$, then $\Psi\in \mathcal{E}(E)$. In this case, we have the following result.
 \begin{lem}\label{lem:restriction}Let $E$ be an isolated pseudo-Levi subgroup of $G$. Then
 the M\"obius function $\mu_{\mathcal{E}(E)}$ is the restriction function of $\mu_{\mathcal{E}(G)}$ on $\mathcal{E}(E) \times \mathcal{E}(E)$.
 \end{lem}
 \begin{proof}
Please see the proof of Corollary A1 in \cite{neil1989remarks}.
 \end{proof}
 
 \begin{rem}It is obvious that we can apply Galois connection theorem (Lemma \ref{lem:galoisconnection}) to $\mathcal{E}(E)$ and $\mathcal{L}(E)$. Note that although $\Psi $ is a Levi subsystem of $\Phi_E$, but it does not need to be a Levi subsystem in $\Phi$ since if $\Psi=\Phi_E\cap U$ for a subspace $U$, but $\Phi\cap U $ might be strictly larger than $\Psi$.
 \end{rem}

\subsection{Main result} \label{s:firstmainresult}
 Under the generic condition, recall our result Equation \eqref{eq:resultmultiplicity}:
\[
   \langle \Lambda \otimes \chi_{\theta_1}\otimes \ldots \otimes \chi_{\theta_{\ell}},1\rangle=  \sum_{\tau \in \mathcal{T}(G)}  a_\tau \sum_{\substack{E\in \mathcal{E}(G)\\ E\text{ is isolated}}}  \mu_{\mathcal{E}(G)}(\Phi_L,\Phi_E)|Z_E^F|,\]
   where $$a_\tau=a_{([\Phi_L],[u])} \colonequals \frac{q^{g\dim(H_\tau)} Q_T^L(u)^{\ell}   |W|^{\ell-1}|[\Phi_L]|}{|C_L(u)^F||W_L|^{\ell-1}}.$$    
To get our result, we need the following lemma.
\begin{lem}
For any a pseudo-Levi subgroup $L$ of $G$, we have \[
 \sum_{\substack{E\in \mathcal{E}(G)\\ E\text{ is isolated}}}  \mu_{\mathcal{E}(G)}(\Phi_L,\Phi_E)|Z_E^F|=  \sum_{\substack{E\in \mathcal{E}(G)\\ E\text{ is isolated}}}\mu_{\mathcal{L}(E)}(\Phi_L,\Phi_E)  N(E) ,
 \]
 where $N(E)=|\{s\in G^F\,|\, C_G(s)=E\}|$.
\end{lem}
For convenience, we define $\mu_{\mathcal{L}(E)}(\Phi_L,\Phi_E)=0$ if $\Phi_L$ is not Levi in $\Phi_E$.
\begin{proof}
Since $s\in Z_E^F \iff E\subset C_G(s)$, we have that 
\[
|Z_E^F|=\sum_{E\leq K \leq G}N(K),
\]
where $H\leq G$ means that $H$ is a subgroup of $G$.
Then we have
\[
\begin{split}
\sum_{\substack{E\in \mathcal{E}(G)\\ E\text{ is isolated}}}  \mu_{\mathcal{E}(G)}(\Phi_L,\Phi_E)|Z_E^F|&=\sum_{\substack{E\in \mathcal{E}(G)\\ E\text{ is isolated}}}  \mu_{\mathcal{E}(G)}(\Phi_L,\Phi_E)\sum_{E\leq K \leq G}N(K)
\\
&=\sum_{\substack{K\in \mathcal{E}(G)\\ K\text{ is isolated}}}  N(K)\sum_{\substack{E\leq K  \\ E\text{ is isolated}}}\mu_{\mathcal{E}(G)}(\Phi_L,\Phi_E) .
\end{split}
\]
Then  we have
\[
\sum_{\substack{E\leq K  \\ E\text{ is isolated}}}\mu_{\mathcal{E}(G)}(\Phi_L,\Phi_E) =\sum_{\substack{E\leq K  \\ E\text{ is isolated}}}\mu_{\mathcal{E}(K)}(\Phi_L,\Phi_E)
= \mu_{\mathcal{L}(K)}(\Phi_L,\Phi_K),
\]
where the first equality is the restriction theorem (Lemma \ref{lem:restriction}), and the last equality is   the Galois connection theorem (Lemma \ref{lem:galoisconnection}). In summary, we get that
\[
\begin{split}
\sum_{\substack{E\in \mathcal{E}(G)\\ E\text{ is isolated}}}  \mu_{\mathcal{E}(G)}(\Phi_L,\Phi_E)|Z_E^F| 
&=\sum_{\substack{K\in \mathcal{E}(G)\\ K\text{ is isolated}}}  \mu_{\mathcal{L}(K)}(\Phi_L,\Phi_K)N(K),
\end{split}
\]
and so we are done.
\end{proof}
With this lemma, we have the following result:
 
\[\begin{split}
   \langle \Lambda \otimes \chi_{\theta_1}\otimes \ldots \otimes \chi_{\theta_{\ell}},1\rangle &=  \sum_{\tau \in \mathcal{T}(G)} \alpha_\tau\sum_{\substack{E\in \mathcal{E}(G)\\ E\text{ is isolated}}}  \mu_{\mathcal{E}(G)}(\Phi_L,\Phi_E)|Z_E^F|\\
   &= \sum_{\substack{E\in \mathcal{E}(G)\\ E\text{ is isolated}}} N(E)\sum_{\tau \in \mathcal{T}(G)} \alpha_\tau \mu_{\mathcal{L}(E)}(\Phi_L,\Phi_E) 
   \end{split}
  \]
  We can easily see that each term 
  $\sum_{\tau \in \mathcal{T}(G)} \alpha_\tau \mu_{\mathcal{L}(E)}(\Phi_L,\Phi_E) $ is related to the size of corresponding generic additive character variety over $\mathfrak{e}=\mathrm{Lie}(E)$ from Equation \eqref{eq:additive} as follows:
  \begin{equation}\label{eq:add and mul}
  \sum_{\tau \in \mathcal{T}(G)} \alpha_\tau \mu_{\mathcal{L}(E)}(\Phi_L,\Phi_E) = \frac{|W|^{\ell-1}}{q^{\gamma_E}|(Z_E^\circ)^F| |W_E|^{\ell-1}} \cdot |\mathcal{X}_\mathfrak{e}^F|
  \end{equation}
  for some  generic regular semisimple adjoint orbits $O_{\mathfrak{e}}$  in $\mathfrak{e}$.
  Remind that $ \mu_{\mathcal{L}(E)}(\Phi_L,\Phi_E)$ is zero if $\Phi_L$ is not Levi subsystem of $\Phi_E$, and any Levi subsystem of $\Phi_E$ is a pseudo-Levi subsystem of $\Phi$ from Lemma \ref{lem:levi-E-pseudo-levi-G}.

\subsubsection{Conclusion}
We are now ready to present our main result by using the discussions and Equation \eqref{eq:add and mul}.
\begin{thm}\label{thm:main}
If $(\theta_1,\ldots ,\theta_{\ell})$ is a generic tuple of regular linear characters of $(T^F)^\ell$, we have that
   \begin{equation}\label{eq:mainresult}
  \langle \Lambda \otimes \chi_{\theta_1}\otimes \ldots \otimes \chi_{\theta_{\ell}},1\rangle    = \sum_{\substack{E\in \mathcal{E}(G)\\ E \text{ is isolated} }} 
    \frac{N(E) |W|^{\ell-1}}{q^{\gamma_E}|(Z_E^\circ)^F| |W_E|^{\ell-1}}\cdot  |\mathcal{X}_\mathfrak{e}^F|,
\end{equation}
where $\gamma_E= \dim(Z_E) +(g-1)\dim(\mathfrak{e}) + \ell|\Phi_E^+|= \dim(Z_G) +(g-1)\dim(\mathfrak{e}) + \ell|\Phi_E^+|$.
In addition, when $g=0$, we have
\[
A_{G^F,B\times \cdots \times B}=  \langle   \chi_{\theta_1}\otimes \ldots \otimes \chi_{\theta_{\ell}},1\rangle    = \sum_{\substack{E\in \mathcal{E}(G)\\ E \text{ is isolated} }} 
    \frac{N(E) |W|^{\ell-1}}{q^{\gamma_E}|(Z_E^\circ)^F| |W_E|^{\ell-1}}\cdot  |\mathcal{X}_\mathfrak{e}^F|
\]
\end{thm}
 
It is easy to see that the leading coefficient of this multiplicity is the leading coefficient of $|Z_G^F|$, and its degree is $(g-1)\dim(G)+\dim(Z_G)+\ell|\Phi^+|$.

\begin{rem}
Note that if there exists a generic linear tuple $(\theta_1,\ldots , \theta_\ell)$ in $L_1^F\times \ldots \times L_\ell^F$ for some parabolic subgroups $P_i=L_iU_i$, then we can compute $\langle \otimes_{i=1}^\ell R_{L_i}^G(\theta_i),1\rangle$ with the same method. This computation is left for interested readers.
\end{rem}

\subsubsection{Example}
Let us consider an example. Let $G=\Sp_4$, $\fg=\mathfrak{sp}_4$ and $E=\SL_2\times \SL_2$. Note that $\SL_2\times \SL_2$ is the only non-trivial isolated pseudo-Levi subgroup in $\Sp_4$. Then
\begin{enumerate}
    \item[(i)] \[
\langle \chi_{\theta_1}\otimes \chi_{\theta_2}\otimes \chi_{\theta_3},1\rangle=2q^2+12q+48,
\]
\item[(ii)]\[
\frac{2}{q^{\gamma_G}}|(\mathcal{X}_{(O_{1,\fg}, O_{2,\fg},O_{3,\fg})}^\fg)^F|= \frac{2}{q^2}\cdot( q^4+6q^3+20q^2)=2q^2+12q+40,
\]
\item[(iii)]\[
\begin{split}
  &   \frac{ |W|^2}{  |W(E)|^2}\cdot 
\frac{N(E)}{q^{\gamma_E} |(Z_E^\circ)^F|} 
|(\mathcal{X}_{(O_{1,Lie(E)}, O_{2,Lie(E)},O_{3,Lie(E)})}^{Lie(E)})^F| =2^2\cdot 2=8.
\end{split}
\]

\end{enumerate}
 Therefore, we can check that Theorem \ref{thm:main} holds in this case.

\subsection{Generic additive character}\label{ss:generic-character}

In this section, we extend \cite[Lemma 6.8.4]{letellier2013quiver} to arbitrary reductive groups, i,e., an additive version of generic character with a generic regular semisimple adjoint orbits of $G$. This is an easy extension, but we give the proof for reader's convenience.  Actually, this is not directly related to our result, but this observation might imply that the computation of the size of an additive character variety in \cite[\S6]{KNWG} is related to compute the number of absolutely indecomposable parabolic $G$-bundles. Note that the generic linear additive character also appears in \cite[\S4.3.2]{Let16}. 

Let us start with defining a Lie algebra version of generic characters of $T^F$. For a Levi subgroup $L$ and its Lie algebra $\mathfrak{l}$, we denote the centre of $\mathfrak{l}$ as $z_\mathfrak{l}$. We say that a linear additive character of $z_\mathfrak{l}^F$ is \emph{generic} if its restriction to $z_\mathfrak{g}^F$ is trivial and its restriction to $z_\mathfrak{m}^F$ is non-trivial for any proper Levi subgroup $M$ of $G$ which contains $L$. This definition is similar to the generic character of $T^F$.
 
 Note that when $\mathcal{O}_1, \ldots , \mathcal{O}_{\ell}$ are adjoint orbits of split regular semisimple elements $x_1, \ldots , x_{\ell}$ in $\fg$, from \cite{letellier2005fourier}, we have the following Fourier transform: $$\mathcal{F}^\fg(\mathbf{X}_{IC_{\overline{\mathcal{O}_i}}})=\mathfrak{R}_{\mathfrak{t}}^\fg(\mathcal{F}^{\mathfrak{t}}(1_{x_i})),$$ where $\mathfrak{t}$ is the Lie algebra of a split torus $T$ in $G$ and $\mathfrak{R}_\mathfrak{t}^\fg$ the Deligne-Lusztig induction on Lie algebra.
Recall that $\mathcal{F}^{\mathfrak{t}}(1_{x_i})\,:\, \mathfrak{t}^F\rightarrow \mathbb{F}_q^\times$ given by $h \rightarrow \Psi(\mu(x_i,h))$, where $\Psi $ is non-trivial additive character of $\mathbb{F}_q$, and $\mu$ is a $G$-invariant non-degenerate bilinear form on $\fg$.
 \begin{lem}
    Let us assume that $({O}_1, \ldots , {O}_{\ell})$ is a tuple of generic regular split semisimple adjoint orbits in $\fg$. Then $\prod_{i=1}^{\ell} \mathcal{F}^{\mathfrak{t}}(1_{\Ad_{g_i}(x_i)})|_{z_{\mathfrak{l}}}$ is a generic character of $z_{\mathfrak{l}}^F$ for any $F$-stable Levi subgroup $L$ of $G$ which satisfies the following condition: For all $i \in \{1,2,\ldots, \ell\}$, there exists $g_i \in G^F$ such that $g_i T g_i^{-1} \subset L$, i.e., $Z_L\subset g_iTg_i^{-1}$.
\end{lem}
 \begin{proof}
    \begin{enumerate}
    \item Let $L$ be a proper Levi subgroup containing $g_1 Tg_1^{-1}, \ldots , g_{\ell} T g_{\ell}^{-1}$, so $\Ad_{g_i}(x_i) \in \mathfrak{l}$ for all $i=1,2, \ldots , \ell$. From the definition of generic (cf. \S\ref{sss:algebra-parabolic-generic}), we have 
    $$\sum_{i=1}^{\ell} \Ad_{g_i}(x_i) \notin [\mathfrak{l}, \mathfrak{l}] \Rightarrow \sum_{i=1}^{\ell} \Ad_{g_i}(x_i)= h_1+h_2 \ \text{such that}\ h_1 \in [\mathfrak{l},\mathfrak{l}],\ h_2 \in z_\mathfrak{l}\backslash \{0\}$$ from the decomposition $\mathfrak{l}=[\mathfrak{l}, \mathfrak{l}] \oplus z_{\mathfrak{l}}$. 
    
    Let us show that $\prod_{i=1}^{\ell} \mathcal{F}^{\mathfrak{t}}(1_{\Ad_{g_i}(x_i)})|_{z_{\mathfrak{l}}}$ is non-trivial on $z_{\mathfrak{l}}^F$.
    Recall that the restriction of $\mu$ to $z_{\mathfrak{l}}$ is also non-degenerated using \cite[Lemma 2.5.14 and 2.5.16]{letellier2005fourier}.
     With this observation, this implies that if $\prod_{i=1}^{\ell} \mathcal{F}^{\mathfrak{t}}(1_{\Ad_{g_i}(x_i)})|_{z_{\mathfrak{l}}}$ is trivial, then we have \begin{equation}
   \begin{split}
    1=&\prod_{i=1}^{\ell} \mathcal{F}^{\mathfrak{t}}(1_{\Ad_{g_i}(x_i)})|_{z_{\mathfrak{l}}}(z)=\prod_{i=1}^{\ell} \Psi(\mu(\Ad_{g_i}(x_i),z))
     \\= &\Psi \left(\mu\left(\sum_{i=1}^{\ell} \Ad_{g_i}(x_i),z\right) \right) 
     =\Psi(\mu(h_1+h_2,z))\\
     =& \Psi(\mu(h_1,z)+\mu(h_2,z))=\Psi(\mu(h_2,z)).
\end{split}
\end{equation} for any $z\in z_{\mathfrak{l}}$. Note that $\mu(h_1,z)=0$ from the fact that the vector space $[\mathfrak{l},\mathfrak{l}]$ is the orthogonal complement of $z_{\mathfrak{l}}$ with respect to $\mu$ using \cite[Lemma 2.5.14 and 2.5.16]{letellier2005fourier}. 

Again, $\mu$ is non-degenerate on $z_{\mathfrak{l}}$, so $\mu(h_2,z)=0$ for any $z\in z_{\mathfrak{l}}$, and this implies that $h_2=0$. Therefore, $\sum_{i=1}^{\ell} \Ad_{g_i}(x_i)= h_1+h_2 =h_1\in [\mathfrak{l},\mathfrak{l}]$, and this contradicts the assumption that $\sum_{i=1}^{\ell} \Ad_{g_i}(x_i) \notin [\mathfrak{l}, \mathfrak{l}]$. In conclusion, $\prod_{i=1}^{\ell} \mathcal{F}^{\mathfrak{t}}(1_{\Ad_{g_i}(x_i)})|_{z_{\mathfrak{l}}}$ is non-trivial.

    \item Let us show that $\prod_{i=1}^{\ell} \mathcal{F}^{\mathfrak{t}}(1_{\Ad_{g_i}(x_i)})|_{z_{\fg}}$ is trivial character on $z_{\fg}^F$. From \cite[Lemma 2.5.16]{letellier2005fourier}, $[\fg,\fg] \perp z_{\fg}$ with respect to $\mu$, i.e., $\mu (h_1,h_2)=0$ when $h_1 \in [\fg,\fg]$ and $h_2\in z_{\fg}$. So, from the assumption that $\sum_{i=1}^{\ell} \Ad_{g_i}(x_i)\in [\fg,\fg]$, we have 
    \[
 \prod_{i=1}^{\ell} \mathcal{F}^{\mathfrak{t}}(1_{\Ad_{g_i}(x_i)})|_{z_{\mathfrak{g}}}(z)=\prod_{i=1}^{\ell} \Psi(\mu(\Ad_{g_i}(x_i),z))=\Psi\left(\mu\left(\sum_{i=1}^{\ell} \Ad_{g_i}(x_i),z\right)\right)=1
    \]
    for any $z\in z_{\fg}$. Therefore, the character $\prod_{i=1}^{\ell} \mathcal{F}^{\mathfrak{t}}(1_{\Ad_{g_i}(x_i)})|_{z_{\fg}}$ is trivial character on $z_{\fg}^F$.
\end{enumerate}
With these two observations, we are done.\end{proof}

  \appendix
  \section{Example}\label{s:example}In this section, we present an example illustrating that if $\mathrm{Ad}_{g_1  }(s_1) + \mathrm{Ad}_{g_2  }(s_2) + \mathrm{Ad}_{g_3  }(s_3)
= 0$, then there exists a subset $t \in \mathrm{Stab}_G(g_1B,g_2B,g_3B)^{ss}$ such that $C_G(t)^\circ$ is an isolated pseudo-Levi subgroup, distinct from \(G\).

We consider the group \(G = \SO_5 = \{ g \in \GL_5 \mid gg^T = g^Tg = 1 \}\) and its Lie algebra \(\mathfrak{g} = \mathfrak{so}_5\), defined over an algebraically closed field \(k\). Our goal is to compute the set $\mathrm{Stab}_G(g_1B, g_2B, g_3B)$ in this setting. Although \(\SO_5\) is not simply connected, it remains meaningful to study this case.

The subgroup $T$ of $G$ is a split maximal torus, i.e., 
\[
T\colonequals \left\{\begin{bmatrix}
	\alpha &\beta&0&0&0\\ -\beta&\alpha&0&0&0 \\ 0&0&\gamma&\omega&0\\ 0&0&-\omega&\gamma&0\\ 0&0&0&0&1
\end{bmatrix} \,\middle\vert\,\substack{ \alpha,\beta,\gamma,\omega\in k\\ \alpha^2+\beta^2=1\\  \gamma^2+\omega^2=1 }\right\}
\]
and its Lie algebra $\ft$ is
\[
\ft= \left\{\begin{bmatrix}
	0 &a&0&0&0\\ -a&0&0&0&0 \\ 0&0&0&b&0\\ 0&0&-b&0&0\\ 0&0&0&0&1
\end{bmatrix} \,\middle\vert\, a,b\in k \right\}.
\]

Let $(O_1, O_2,O_3)$ be a generic regular semisimple tuple in $\fg$ such that $O_i=\Ad_G(s_i)$ for some $s_i \in \ft$ for all $i=1,2,3$. Then let us take an element $g_1,g_2,g_3\in G$ such that
\[
\Ad_{g_1}(s_1)+\Ad_{g_2}(s_2)+\Ad_{g_3}(s_3)=0.
\]
Then our purpose is to find some semisimple elements in 
\[
\mathrm{Stab}_G((g_1B,g_2B,g_3B)).
\]

\subsection{An example of generic tuple}
We will give an example of generic regular semisimple adjoint orbits.
To check the genericity, we  will use Lemma \ref{lem:equivalence}.

Let
\[
s_1=\begin{bmatrix}
	0&3&0&0&0\\-3&0&0&0&0\\0&0&0&6&0\\0&0&-6&0&0\\0&0&0&0&0\end{bmatrix}
, \quad s_2=\begin{bmatrix}
	0&9&0&0&0\\ -9&0&0&0&0\\0&0&0&18&0\\0&0& -18&0&0\\0&0&0&0&0
\end{bmatrix}\text{ and }s_3=\begin{bmatrix}
	0&-8\sqrt{6}&0&0&0\\8\sqrt{6}&0&0&0&0\\0&0&0&-4\sqrt{6}&0\\0&0& 4\sqrt{6}&0&0\\0&0&0&0&0
\end{bmatrix}.\]
These are elements in $\fg$ since $\fg$ consists of skew-symmetric matrices.
Note that over $\GL_5$, $s_1,s_2$ and $s_3$ are diagonalisable, so every element in $C_G(s_i)$  is semisimple. This implies that $s_1,s_2$ and  $s_3$ are regular.  

\subsection{Genericity} Let us check that why this is generic with Lemma \ref{lem:equivalence}. The Weyl group of $SO_5$, denoted by $W$, is genereated by two elements, say $w_1$ and $w_2$. We have that
\[
w_1\cdot \begin{bmatrix}
	0 &a&0&0&0\\ -a&0&0&0&0 \\ 0&0&0&b&0\\ 0&0&-b&0&0\\ 0&0&0&0&1
\end{bmatrix}= \begin{bmatrix}
	0 &b&0&0&0\\ -b&0&0&0&0 \\ 0&0&0&a&0\\ 0&0&-a&0&0\\ 0&0&0&0&1
\end{bmatrix}
\]
and 
\[
w_2\cdot \begin{bmatrix}
	0 &a&0&0&0\\ -a&0&0&0&0 \\ 0&0&0&b&0\\ 0&0&-b&0&0\\ 0&0&0&0&1
\end{bmatrix}= \begin{bmatrix}
	0 &-a&0&0&0\\ a&0&0&0&0 \\ 0&0&0&b&0\\ 0&0&-b&0&0\\ 0&0&0&0&1
\end{bmatrix}.
\]

Note that $s=\begin{bmatrix}
	0 &a&0&0&0\\ -a&0&0&0&0 \\ 0&0&0&b&0\\ 0&0&-b&0&0\\ 0&0&0&0&1
\end{bmatrix}\in \ft$ is not in $[\fl,\fl]$ for any proper standard Levi subgroup if and only if $a=0$, $b=0$ or $a=b$.
This is because every proper Levi subalgebra in $\fg$ has one semisimple rank, i.e., $[\fl,\fl]\simeq \mathfrak{sl}_2$ for any proper Levi subalgebra $\fl$. From direct computation, (I will skip the computation,, but please let me know if you want)  we can check that $\ft\cap [\fl,\fl]$ is the form of
\[\begin{bmatrix}
	0&a&0&0&0\\-a&0&0&0&0\\0&0&0&0&0\\0&0&0&0&0\\0&0&0&0&0
\end{bmatrix}\ \text{or }
\begin{bmatrix}
	0&a&0&0&0\\-a&0&0&0&0\\0&0&0&\pm a&0\\0&0&\mp a&0&0\\0&0&0&0&0
\end{bmatrix}.
\]

With the above discussions, it is easy to check that $(s_1,s_2,s_3)$ is generic.
This is because    $a+8\sqrt{6}\neq 0$ and $a\pm 8\sqrt{6}\neq b\pm 4\sqrt{6}$ for any integers $a$ and $b$.
\subsection{Finding an example of $(v_1,v_2,v_3)$}
Let us take as follows:
\[
g_1=\begin{bmatrix}
	1&0&0&0&0\\ 0&1/3&2/3&2/3&0\\0&2/3&-2/3&1/3&0\\0&2/3&1/3&-2/3&0\\0&0&0&0&1
\end{bmatrix},\quad 
g_2=\begin{bmatrix}
	1/3&0&2/3&2/3&0\\0&1&0&0&0\\2/3&0&-2/3&1/3&0\\2/3&0&1/3&-2/3&0\\0&0&0&0&1
\end{bmatrix}
\]
and \[\widetilde{g_3}=\begin{bmatrix}
	\frac{1}{2}\sqrt{\frac{3}{14}}(\sqrt{6}i-1)&-	\frac{1}{2}\sqrt{\frac{3}{14}}(\sqrt{6}i-1)&\frac{3-i\sqrt{6}}{2\sqrt{30}}
	&\frac{3+i\sqrt{6}}{2\sqrt{30}}&0 \\
	\frac{\sqrt{6}i-1}{2\sqrt{42}} &-\frac{\sqrt{6}i+1}{2\sqrt{42}}&\frac{1+3i\sqrt{6}}{2\sqrt{30}}&\frac{1-3i\sqrt{6}}{2\sqrt{30}}&0\\
	\frac{-5-2i\sqrt{6}}{2\sqrt{42}}&\frac{-5+2i\sqrt{6}}{2\sqrt{42}}&\frac{\sqrt{5}}{2\sqrt{6}}&\frac{\sqrt{5}}{2\sqrt{6}}&0 \\
	\frac{\sqrt{7}}{2\sqrt{6}}&\frac{\sqrt{7}}{2\sqrt{6}}&\frac{\sqrt{5}}{2\sqrt{6}}&\frac{\sqrt{5}}{2\sqrt{6}}&0 \\
	0&0&0&0&1
\end{bmatrix}\cdot \begin{bmatrix}
	0&0&\frac{-i}{\sqrt{2}}&\frac{i}{\sqrt{2}}&0\\
	0&0&\frac{1}{\sqrt{2}}&\frac{1}{\sqrt{2}}&0\\
	\frac{-i}{\sqrt{2}}&\frac{i}{\sqrt{2}}&0&0&0\\
	\frac{1}{\sqrt{2}}&\frac{1}{\sqrt{2}}&0&0&0\\
	0&0&0&0&1
\end{bmatrix}^{-1}.
\]
Then we have that
\[
\Ad_{g_1}(s_1)=\begin{bmatrix}
	0&1&2&2&0\\-1&0&4&-4&0\\-2&-4&0&2&0\\-2&4&-2&0&0\\0&0&0&0&0
\end{bmatrix},\quad \Ad_{g_2}(s_2)=\begin{bmatrix}
	0&3&12&-12&0\\-3&0&-6&-6&0\\-12&6&0&6&0\\12&6&-6&0&0\\0&0&0&0&0
\end{bmatrix}
\]
and
\[\Ad_{\widetilde{g_3}}(s_3)=\begin{bmatrix}
	0&-4&-14&10&0\\
	4&0&2&10&0\\
	14&-2&0&-8&0\\
	-10&-10&8&0&0\\
	0&0&0&0&0
\end{bmatrix}.
\]

Then we can check that 
\[
\Ad_{g_1}(s_1)+\Ad_{g_2}(s_2)+\Ad_{\widetilde{g_3}}(s_3)=0.
\]

Note that $\widetilde{v_3}\notin \SO_5$.  However, from \cite[Lemma 3.2]{AL11}, we can take $g_3 \in SO_5$ such that $s_3'\colonequals \Ad_{g_3}(s_3)=\Ad_{\widetilde{g_3}}(s_3)$. Furthermore, from the form of $s_3$ and $s_3'$, we can deduce that $g_3\in \SO_4\times k \subset SO_5$, i.e., $g_3$ and $\widetilde{g_3}$ have the same form.
\subsection{Stabiliser}
Let us take \[
t=\begin{bmatrix}
	-1&0&0&0&0\\0&-1&0&0&0\\0&0&-1&0&0\\0&0&0&-1&0\\0&0&0&0&1
\end{bmatrix}\in T\subset G.\]
Then we can easily check that 
\[
t(g_iB)=t(g_it^{-1}B)=(tg_it^{-1})B=g_iB\quad \text{for all}\ i=1,2,3
\]from the form of $g_1,g_2$ and $g_3$.
Therefore, $t\in \mathrm{Stab}_G((g_1T,g_2T,g_3T)).$
Recall that $G=\SO_5$ is adjoint, i.e., it has only the trivial centre, and so $t$ is not a central element. 
Therefore, this implies that there exists $
t\in \mathrm{Stab}_G((g_1B,g_2B,g_3B))$ {such that} $C_G(t)$ {is proper isolated pseudo-Levi subgroup}.

 \section*{Acknowledgments}		
GyeongHyeon Nam  was supported by the National Research Foundation of Korea (NRF) grant funded by the Korea government (MSIT) (No. RS-2024-00334558 and No. RS-2025-02262988), Oscar Kivinen's Väisälä project grant of the Finnish Academy of Science and Letters and a KIAS Individual Grant
(HP111501) at Korea Institute for Advanced Study.
 
The author is very grateful to Emmanuel Letellier for his helpful comments and suggestions on this project, and would also like to thank Oscar Kivinen for his valuable feedback.

\end{document}